%% file: SllG_Multi_Ver10.tex
\documentclass[11pt]{amsart}
\usepackage{a4wide}
\usepackage{amsmath,amssymb,amsthm,latexsym,amsbsy,amsfonts}
\usepackage{color,verbatim}
\usepackage{pstricks}
\usepackage{tabularx}
\usepackage{rotating}
\usepackage{ulem}
\usepackage{graphicx,epsfig}
\usepackage[colorlinks,linkcolor={blue},
citecolor={blue},urlcolor={red},]{hyperref}
\usepackage{hyperref}

\input{ams_notation}

\input{env}
\title[FEM for stochastic Landau--Lifshitz--Gilbert
equation]{A FINITE ELEMENT APPROXIMATION FOR
THE STOCHASTIC LANDAU--LIFSHITZ--GILBERT EQUATION WITH MULTI-DIMENSIONAL NOISE}
\author{Beniamin Goldys}
\address{School of Mathematics and Statistics,
         The University of Sydney,
         Sydney 2006, Australia}
\email{beniamin.goldys@sydney.edu.au}
\author{Joseph Grotowski}
\address{School of Mathematics and Physics,
         The University of Queensland,
         QLD 4072, Australia}
\email{j.grotowski@uq.edu.au}
\author{Kim-Ngan Le}
\address{School of Mathematics and Statistics,
         The University of New South Wales,
         Sydney 2052, Australia}
\email{n.le-kim@unsw.edu.au}

\date{\today}

\newtheorem{algorithm}{Algorithm}[section]

\begin{document}
\begin{abstract}
We propose an unconditionally 
convergent linear finite element scheme for the 
stochastic Landau--Lifshitz--Gilbert (LLG) equation with multi-dimensional noise.
By using the Doss-Sussmann technique, we first transform the stochastic LLG equation 
into a partial differential equation that depends on the solution of 
the auxiliary equation for the diffusion part.
The resulting equation has solutions absolutely continuous with respect to time. 
We then propose a convergent
$\theta$-linear scheme for the numerical solution of the
reformulated equation. As a consequence,
we are able to show the existence of weak
martingale solutions to the stochastic LLG equation.

\end{abstract}
\maketitle
\tableofcontents
\section{Introduction}
The deterministic Landau-Lifschitz-Gilbert (LLG) equation provides a basis for  the theory and applications of ferromagnetic materials and fabrication of magnetic memories in particular, see for example \cite{Chantrelletal2015,Cimrak_survey,Gil55,LL35}. 
Let us recall, that in this theory we consider a ferromagnetic material filling the domain $D$ and a function  $\vecu\in H^{1,2}\left(D,\mathbb S^2\right)$, where $\mathbb S^2$ stands for the unit sphere in $\mathbb R^3$, represents a configuration of magnetic moments across the domain $D$, that is $\vecu(x)$ is the magnetisation vector at the point $x\in D$. According to the Landau and Lifschitz  theory of ferrormagnetizm \cite{LL35}, modified later by Gilbert \cite{Gil55}, the time evolution of magnetic moments $\vecM(t,x)$ is described, in the simplest case, by the Landau-Lifschitz-Gilbert (LLG) equation
\begin{equation}\label{equ:llg}
\begin{cases}
&\dfrac{\partial\vecM}{\partial t}
=
\lambda_1\vecM\times\Delta \vecM
-
\lambda_2\vecM\times(\vecM\times\Delta\vecM)
\quad\text{ in } (0,T)\times D,\\
\\
&\frac{\pa\vecM}{\pa \vecn}
= 0 \quad\text{ in }  (0,T)\times \pa D,\\
\\
&\vecM(0,\cdot)
=
\vecM_0(\cdot) \quad\text{ in } D, 
\end{cases}
\end{equation}
where $\lambda_1\not=0$ and $\lambda_2 >0$ 
are constants, and  $\vecn$ stands for the outward normal vector on $\partial D$; see
e.g. \cite{Cimrak_survey}. We assume that $\vecM_0\in H^{1,2}\left( D,\mathbb S^2\right)$, and then one can show that 
\begin{equation}\label{s2}
|\vecM(t,x)|=1,\quad t\in[0,T],\,\,x\in D
\end{equation}
\par
In this paper we are concerned with a stochastic version of the LLG equation. Randomly fluctuating fields were  originally introduced in physics by N\'eel in \cite{neel} as formal quantities responsible for magnetization fluctuations. The necessity 
of being able to describe deviations from the average magnetization trajectory in an ensemble of noninteracting 
nanoparticles was later emphasised by Brown in ~\cite{Brown1979,Brown63}. According to a non-rigorous arguments of Brown the magnetisation $\vecM$ evolves randomly according to a stochastic version of \eqref{equ:llg} that takes the form, (see~\cite{BrzGolJer12} for more details about the physical background and derivation of this equation) 
\begin{align}\label{E:1.1}
\begin{cases}
&d\vecM
=
\big(\lambda_1 \vecM\times \Delta \vecM
-
\lambda_2 \vecM\times(\vecM\times \Delta \vecM)\big)dt
+
\sum_{i=1}^q(\vecM\times \vecg_i)\circ dW_i(t),\\
\\
&\dfrac{\pa\vecM}{\pa \vecn}
= 0 \quad\text{ on } (0,T) \times \pa D,\\
\\
&\vecM(0,\cdot)
=
\vecM_0 \quad\text{ in } D, 
\end{cases}
\end{align}
where $\vecg_i \in\mW^{2,\infty}(D)$, $i=1,\cdots,q$, satisfy the homogeneous 
Neumann boundary conditions and $\left(W_i\right)_{i=1}^q$ is a $q$-dimensional Wiener process. In view of the property~\eqref{s2}
for the deterministic system, we require that $\vecM$ also
satisfies~\eqref{s2}. To this end we are forced to use the Stratonovich
differential $\circ\: dW_i(t)$ in equation \eqref{E:1.1}. Mathematical theory of equation \eqref{E:1.1} has been initiated only recently, in ~\cite{BrzGolJer12}, where the existence of weak martingale solutions to \eqref{E:1.1} was proved for the case $q=1$ using the Galerkin-Faedo approximations. Let us note, that usually the Galerkin-Faedo approximations do not provided a useful computational tool for solving an equation. 
\par\bigskip
The aim of this paper is two-fold. We will prove the existence of solutions to the stochastic LLG equation \eqref{E:1.1} and at the same time will provide an efficient and flexible algorithm for solving numerically this equation. To this end we will use the finite element method and a new transformation of the Stratonovich type equation \eqref{E:1.1} to a deterministic PDE \eqref{InE:14} with coefficients determined by a stochastic ODE \eqref{auE:2} that can be solved separately. The deterministic PDE we obtain, has solutions absolutely continuous with respect to time, hence convenient for the construction of a convergent finite elemetn scheme. Our approach is based on the 
Doss-Sussmann technique~\cite{Doss1977,sussmann1977}. This transformation was introduced in \cite{BNT2016} to study the stochastic LLG equation with a single Wiener process ($q=1$), in which case the auxiliary ODE is deterministic. Since the vector fields $\vecu\times\vecg_i$ are non-commuting, the case of $q>1$ is more difficult and requires new arguments. 
\par
We apply the finite element method to the PDE resulting from this transformation and prove the convergence of linear finite element scheme to  a weak martingale solution to ~\eqref{E:1.1} (after taking an inverse transformation). Our proof is simpler than the proof in ~\cite{BrzGolJer12} and covers the case of $q>1$. We note here that under appropriate assumptions even the case of infinite-dimensional noise ($q=\infty$) can be handled in exactly the same way. 
\par
Let us recall that the first convergent finite element scheme for the stochastic LLG equation was studied in~\cite{BanBrzPro09} and 
is based on a Crank--Nicolson type time-marching evolution, relying 
on a nonlinear iteration solved by a fixed point method. 
On the other hand, there has been an intensive development of a new class of numerical methods for the LLG equation~\eqref{equ:llg}
based on a linear iterations, yielding unconditional convergence and stability~\cite{Alo08,AloJai06}. 
The ideas developed there are extended and generalized in~\cite{BNT2016,Alouges2014} in order to 
take into account the stochastic term. A fully linear discrete scheme for~\eqref{E:1.1} 
is studied in~\cite{BNT2016} but with one-dimensional noise. The method is based on the so--called 
Doss-Sussmann technique~\cite{Doss1977,sussmann1977}, which allows one to replace the stochastic partial differential equation (PDE) 
by an equivalent PDE with random coefficients. 
In contrast, ~\cite{Alouges2014} considers, for a more  general noise, a projection scheme 
applied directly to the original stochastic equation~\eqref{E:1.1}. However, this  
approach requires a quite specific and complicated treatment of 
the stochastic term.
In this paper, we propose a convergent
$\theta$-linear scheme for the numerical solution of the tranformed equation and prove unconditional stability and
convergence for the scheme when $\theta>1/2$. To the best of our knowledge this is a new result for this problem. 
\par
The paper is organised as follows. In Section~\ref{sec:wea
sol} we define the notion of weak martingale solutions to~\eqref{E:1.1}
and state our main result. In Section~\ref{sec:pre}, we introduce an auxiliary stochastic ODE and prove some  properties of solution necessary for the transformation of 
equation~\eqref{E:1.1} to a deterministic PDE with random coefficients. 
Details of this transformation are presented in
Section~\ref{sec:equ eqn}. We also show in this section how
a weak solution to~\eqref{E:1.1} can be obtained from a
weak solution of the reformulated form.
In Section~\ref{sec:fin ele} we introduce our finite element
scheme and present a proof for the stability of approximate solutions.
Section~\ref{sec:pro} is devoted to the proof of
the main theorem, namely the convergence of finite
element solutions to a weak solution of the reformulated
equation. Finally, in the Appendix we collect, for 
the reader's convenience, a number of facts that are used in the course of the proof.

Throughout this paper, $c$ denotes a generic constant that
may take different values at different occurrences. In what follows we will also use the notation $D_T = (0,T) \times D$. 
\section{Definition of a weak solution and the main
result}\label{sec:wea sol}
\noindent
In this section we state the definition of a weak
solution to~\eqref{E:1.1} and present our main result.
Before doing so, we introduce some suitable Sobolev spaces, and fix some notation. The standing assumption for the rest of the paper is that $D$ is a bounded open domain in $\R^3$ with a smooth boundary. 
\par
For any $U\subset\R^d$, $d\ge 1$, we denote by $\mL^2(U)$ the space of Lebesgue square-integrable functions defined on $U$ and taking values in $\R^3$. 
The function space $\mH^1(U)$ is defined as:
\begin{align*}
\mH^1(U)
&=
\left\{ \vecu\in\mL^2(U) : \frac{\p\vecu}{\p
x_i}\in \mL^2(U)\quad\text{for } i\le d.
\right\}.
\end{align*}
\begin{remark}
For $\vecu, \vecv\in \mH^1(D)$
we denote
\begin{align*}
\vecu\times\nabla\vecv
&:=
\left(
\vecu\times\frac{\partial\vecv}{\partial x_1},
\vecu\times\frac{\partial\vecv}{\partial x_2},
\vecu\times\frac{\partial\vecv}{\partial x_3}
\right) \\
\nabla\vecu\times\nabla\vecv
&:=
\sum_{i=1}^3
\frac{\partial\vecu}{\partial x_i}
\times
\frac{\partial\vecv}{\partial x_i} \\
\inpro{\vecw\times\nabla\vecv}{\nabla\vecu}_{\mL^2(D)}
&:=
\sum_{i=1}^3
\inpro{\vecw\times\frac{\partial\vecv}{\partial x_i}}%
{\frac{\partial\vecu}{\partial x_i}}_{\mL^2(D)}\quad\forall\,\vecw\in\mL^{\infty}(D).
\end{align*}
\end{remark}
\begin{definition}\label{def:wea sol}
Given $T\in(0,\infty)$ and a family of functions $\left\{g_i:\, i=1,\ldots,q\right\}\subset\mathbb L^\infty(D)$, a weak martingale solution
$(\Omega,\cF,(\cF_t)_{t\in[0,T]},\mP,W,\vecM)$
to~\eqref{E:1.1}, for the time interval $[0,T]$,
consists of
\begin{enumerate}
\renewcommand{\labelenumi}{(\alph{enumi})}
\item
a filtered probability space
$(\Omega,\cF,(\cF_t)_{t\in[0,T]},\mP)$ with the
filtration satisfying the usual conditions,
\item
a $q$-dimensional $(\cF_t)$-adapted Wiener process
$W=(W_t)_{t\in[0,T]}$,
\item
a progressively measurable
process $\vecM : [0,T]\times\Omega \goto \mL^2(D)$
\end{enumerate}
such that
\begin{enumerate}
\item
\quad
$\vecM(\cdot,\omega) \in C([0,T];\mH^{-1}(D))$
for $\mP$-a.e. $\omega\in\Omega$;
\item
\quad
$\mE\left(
\esssup_{t\in[0,T]}\|\nabla\vecM(t)\|^2_{\mL^2(D)}
\right) < \infty$;
\item
\quad
$|\vecM(t,x)| = 1$
for each $t \in [0,T]$, a.e. $x\in D$, and $\mP$-a.s.;
\item
\quad
for every  $t\in[0,T]$,
for all $\vecpsi\in\C_0^{\infty}(D)$,
$\mP$-a.s.:
\begin{align}\label{wE:1.1}
\inpro{\vecM(t)}{\vecpsi}_{\mL^2(D)}
-
\inpro{\vecM_0}{\vecpsi}_{\mL^2(D)}
&=
-\lambda_1
\int_0^t
\inpro{\vecM\times\nabla\vecM}{\nabla\vecpsi}_{\mL^2(D)}\ds\nn\\
&\quad-
\lambda_2
\int_0^t
\inpro{\vecM\times\nabla\vecM}{\nabla(\vecM\times\vecpsi)}_{\mL^2(D)}\ds\nn\\
&\quad+
\sum_{i=1}^q
\int_0^t
\inpro{\vecM\times\vecg_i}{\vecpsi}_{\mL^2(D)}\circ dW_i(s).
\end{align}
\end{enumerate}
\end{definition}

As the main result of this paper, we will establish a finite element scheme defined via a sequence 
of functions which are piecewise linear in both the space and time variables. 
We also prove that this sequence contains a subsequence converging to a weak martingale 
solution in the sense of Definition~\ref{def:wea sol}. A precise statement will be given in Theorem~\ref{the:mai}.

\section{The auxiliary equation for the diffusion part}\label{sec:pre}
In this section we introduce the auxiliary equation \eqref{auE:1} that will
be used in the next section to define a new variable from $\vecM$, and establish some properties of its solution.\\
Let $\vecg_1,\ldots,\vecg_q\in C\left(\overline{D},\mathbb R^3\right)$, be fixed. For $i=1,\ldots,q$, and $x\in\overline{D}$ we 
define linear operators $G_i(x):\mathbb R^3\to\mathbb R^3$ by $G_i(x)\vecu=\vecu\times \vecg_i(x)$. 
In what follows we suppress the argument $x$. It is easy to check that 
\begin{align}
G_i^\star=-G_i\,,\label{eq: nabZ2}\\
\quad\text{and}\quad
\left(G_i^2\right)^\star=G_i^2\,.\label{eq: nabZ3}
\end{align}
We will consider a stochastic Stratonovitch equation on the algebra  $\mathcal L\left(\mathbb R^3\right)$ of linear operators in $\mathbb R^3$: 
\begin{equation}\label{eqz}
Z_t=I+\sum_{i=1}^q\int_0^t G_iZ_s\circ dW_i(s),\quad t\ge 0\,.
\end{equation}
\begin{lemma}\label{lemnew1}
Let $\vecg_1,\ldots,\vecg_q\in C\left(\overline{D},\mathbb R^3\right)$. Then the following holds.\\
(a) For every $x\in\overline{D}$ equation \eqref{eqz} has a unique strong solution, which has a $t$-continuous version in  $\mathcal L\left( \mathbb R^3\right)$.\\
(b) For every $t\ge 0$ and $x\in\overline{D}$ 
\begin{equation}\label{iso}
\left|Z_t\vecu\right|=|\vecu|\quad\mathbb P-a.s\quad\mathrm{for\,\,every\,\,}\vecu\in\mathbb R^3.
\end{equation}
In particular, 
 for every $t\ge 0$ the operator $Z_t$ is invertible and $Z_t^{-1}=Z_t^\star$.\\
 (c) If moreover  $\vecg_1,\ldots,\vecg_q\in C^\alpha\left(\overline{D},\mathbb R^3\right)$ for a certain $\alpha\in(0,1)$ then the mapping $(t,x)\to Z_t(x)$ has a continuous version in $\mathcal L\left(\R^3\right)$. 
\end{lemma}
\begin{proof}
Equation \eqref{eqz} can be equivalently written as an It\^o equation
\begin{equation}\label{eqz-i}
Z_t=I+\frac{1}{2}\sum_{i=1}^q\int_0^t G_i^2Z_s\,ds+\sum_{i=1}^q\int_0^t G_iZ_s\, dW_i(s),\quad t\ge 0\,.
\end{equation}
Since the coefficients of equation \eqref{eqz-i} are Lipschitz, the existence and uniqueness of strong solutions to equation \eqref{eqz-i}, and the existence of its continuous version is standard, see for example Theorem 18.3 in \cite{kallenberg}. Hence, the same result holds for \eqref{eqz}. 
\par\noindent
To prove (b) we fix $x\in\overline{D}$, $t\ge 0$ and $\vecu\in\mathbb R^3$ and put $Z^{\vecu}=Z\vecu$. Then equation \eqref{eqz-i} yields 
\begin{equation}\label{auE:2}
Z_t^{\vecu}=\vecu+\frac{1}{2}\sum_{i=1}^q
\int\limits_0^t G_i^2Z_s^{\vecu}\ds
+\sum_{i=1}^q
\int\limits_0^tG_iZ_s^{\vecu} \,dW_i(s)\,.
\end{equation}
Applying the It\^o formula to the process $|Z_t^{\vecu}|^2$ and invoking \eqref{eq: nabZ2} we obtain 
\begin{align*}
d|Z_t^{\vecu}|^2
&=
2\inpro{Z_t^{\vecu}}{dZ_t^{\vecu}} 
+ 
\sum_{i=1}^q
|G_iZ_t^{\vecu}|^2\dt\\
&=
\sum_{i=1}^q
\inpro{Z_t^{\vecu}}{G_i^2Z_t^{\vecu}}\dt
+
2
\sum_{i=1}^q
\inpro{Z_t^{\vecu}}{G_iZ_t^{\vecu}}\,dW_i(t)
+
\sum_{i=1}^q
|G_iZ_t^{\vecu}|^2\dt\\
&=0, 
\end{align*}
or equivalently 
\begin{equation*}
|Z_t^{\vecu}|^2=|\vecu|^2,\quad\mathrm{for\,\,all\,\,}t\ge 0,\quad \mP-a.s..
\end{equation*}

To prove (c), we begin by letting $0\le s<t\le T$ and $x,y\in\overline{D}$. For any $p\ge 1$ we have 
\begin{equation}\label{reg1}
\mathbb E\left|Z_{t}(y)-Z_{s}(x)\right|^p\le 2^{p-1}\mathbb E\left|Z_{t}(y)-Z_{s}(y)\right|^p
+2^{p-1}\mathbb E\left|Z_{s}(y)-Z_{s}(x)\right|^p\,.
\end{equation}
It is well known that there exists $C_1>0$ such that 
\begin{equation}\label{reg2}
\mathbb E\left|Z_{t}(y)-Z_{s}(y)\right|^p\le C_1|t-s|^{\frac{p}{2}}.
\end{equation}
If there exists $\alpha\in(0,1]$ such that 
\[|g_i(x)-g_i(y)|\le c_i|x-y|^\alpha,\quad x,y\in \overline{D},\,\,i=1,\ldots, q\]
then for a certain $C>0$, for any $h\in\R^3$ there holds 
\begin{equation}\label{g1}
\begin{aligned}
|G_i(x)h|&\le C|h|\,,\\
|(G_i(x)-G_i(y))h|&\le C|x-y|^\alpha |h|\,,\\
\left|\left(G_i^2(x)-G_i^2(y)\right)h\right|&\le C|x-y|^\alpha |h|.
\end{aligned}
\end{equation}
Then 
\[\begin{aligned}
Z_{s}(y)-Z_{s}(x)&=\frac{1}{2}\sum_{i=1}^q\int_0^s G_i^2(y)Z_r(y)\,dr
+\sum_{i=1}^q\int_0^sG_i(y)Z_r(y)dW_i(r)\\
&\,\,\,\,\,\,\,-\frac{1}{2}\sum_{i=1}^q\int_0^sG_i^2(x)Z_r(x)\,dr-\sum_{i=1}^q\int_0^sG_i(x)Z_r(x)dW_i(r)\\
&=\frac{1}{2}\sum_{i=1}^q\int_0^s \left(G_i^2(y)-G_i^2(x)\right)Z_r(y)\,dr+\frac{1}{2}\sum_{i=1}^q\int_0^s G_i^2(x)\left(Z_r(y)-Z_r(x)\right)\,dr\\
&\,\,\,\,\,\,\,+\sum_{i=1}^q\int_0^s\left(G_i(y)-G_i(x)\right)Z_r(y)dW_i(r)+\sum_{i=1}^q\int_0^sG_i(x)\left(Z_r(y)-Z_r(x)\right)dW_i(r).
\end{aligned}\]
Using \eqref{g1} we obtain
\[\begin{aligned}
\mathbb E\left|Z_{s}(y)-Z_{s}(x)\right|^p&\le\widetilde{C}|x-y|^{\alpha p}+\widetilde{C}\int_0^s\mathbb E\left|Z_r(y)-Z_r(x)\right|^pdr. 
\end{aligned}\]
Therefore, invoking the Gronwall Lemma we obtain 
\begin{equation}\label{reg3}
\mathbb E\left|Z_{s}(y)-Z_{s}(x)\right|^p\le \widetilde{C}e^{\widetilde{C}T}|x-y|^{\alpha p}.
\end{equation}
Combining \eqref{reg1}, \eqref{reg2} and \eqref{reg3} we obtain 
\begin{equation}\label{reg4}
\mathbb E\left|Z_{t}(y)-Z_{s}(x)\right|^p\le c_1|t-s|^{\frac{p}{2}}+c_2|x-y|^{\alpha p}.
\end{equation}
Let $\beta>0$ and $r=d+1+\beta$. Let $p$ be chosen in such a way that 
\[\frac{p}{2}\ge r\quad\mathrm{and}\quad p\alpha\ge r\,.\]
The set $[0,T]\times D$ can be covered by a finite number of open sets $B_k$ with the property $|t-s|^r+|x-y|^r<1$ on every set $B_k$. In each $B_k$, \eqref{reg4} then  yields 
\[\mathbb E\left|Z_{t}(y)-Z_{s}(x)\right|^p\le c\left(|t-s|^r+|x-y|^r\right)\,,\]
and the result then follows by the Kolmogorv-Chentsov theorem, see p. 57 of \cite{kallenberg}.
 \end{proof}
\begin{lemma}\label{lemnew2}
Assume that $\vecg_i\in C^{1+\alpha}_b(D,\mathbb R^3)$. Then the following holds.\\
(a) For every $t\ge 0$ we have $Z_t\in C^1_b(D,\mathcal L(\mathbb R^3))$ $\mathbb P$-a.s. \\
(b) For every $x\in D$ the process $\xi_t(x)=\nabla Z_t(x)$ is the unique solution of the linear It\^o equation
\begin{equation*}
d\xi_t(x)
=
\frac12\sum_{i=1}^q
\bigl(
G_i^2\xi_t(x) + H_iZ_t(x)
\bigr)\dt
+
\sum_{i=1}^q
\bigl(
G_i\xi_t(x) + I_i Z_t(x)
\bigr)\,dW_i(t),
\end{equation*}
with $\xi_0(x)=0$ and the operators $H_i,I_i\in\mathcal L\left(\mathbb R^3\right)$ defined as 
\begin{align*}
I_i\vecu = \vecu\times\nabla\vecg_i
\quad\text{and}\quad
H_i\vecu = G_i\vecu\times\nabla\vecg_i + G_i(\vecu\times\nabla\vecg_i).
\end{align*}
(c) For every $\gamma<\min\left(\alpha,\frac{1}{2}\right)$ the mapping $(t,x)\to\nabla Z_t(x)$ is $\gamma$-H\"older continuous. \\
(d) We have
\[\mathbb E\sup_{t\le T}\,\sup_{x\in\overline{D}}|\nabla Z_t|^2<\infty\,.\]
\end{lemma}
\begin{proof}
(a) Let $E$ denote the Banach space of continuous and adapted processes $Z$ taking values in the space of  linear operators $\mathcal L\left(\mathbb R^3\right)$ and endowed with the norm 
\[\|Z\|_E=\left(\mathbb E\sup_{t\le T}|Z_t|^2\right)^{1/2}\,.\]
For every $x\in D$ we define a mapping 
\[\mathcal K:D\times E\to D\times E,\quad \mathcal K(x,Z)(t)=I+\sum_{i=1}^q\int_0^tG_i(x)Z_s\circ dW_i(s)\,.\]
It is easy to check that the assumptions of Lemma 9.2, p. 238 in \cite{DaPrato92}
are satisfied and therefore (a) holds.\\
(b) The proof is completely analogous to the proof of Theorem 9.8 in \cite{DaPrato92}, and is hence omitted.\\
(c) The proof is analogous to the proof of part (c) of Lemma \ref{lemnew1}.\\
(d) The estimate follows easily from (c). 
\end{proof}
For every $\vecu\in\mL^2(D)$ we will consider the $L^2(D)$-valued process $Z_t(u)$ defined by 
\[[Z_t(u)](x)=Z_t(x)u(x)\quad x-a.e.\]
Clearly, 
\begin{equation}\label{auE:1}
Z_t(\vecu) 
= 
\vecu
+
\sum_{i=1}^q
\int_0^t
Z_s(\vecu)\times\vecg_i\circ dW_i(s),
\quad t\geq 0\,,
\end{equation}
where the equality holds in $\mL^2(D)$. The process $Z_t$ is now an operator-valued process taking values $\mathcal L\left(\mL^2(D)\right)$ and it will still be denoted by $Z_t$. 
The next lemmas follow immediately from the properties of the matrix-valued process considered above. 
\begin{lemma}\label{lem:auxi1} Assume that $\left\{g_i:\, i=1,\ldots,q\right\}\subset C_b^{1+\alpha}(D)$. Then for every $\vecu\in\mathbb L^2(D)$ the stochastic differential equation \eqref{auE:1} has a unique strong continuous solution in $\mathbb L^2(D)$. 
Moreover, there exists $\Omega_0\subset\Omega$ such that $\mP\left(\Omega_0\right)=1$ and for every $\omega\in\Omega_0$ the following holds.\\
(a) 
For all $t\geq 0$ and every $\vecu\in\mL^2(D)$, 
\[
|Z_t(\omega,\vecu)| = |\vecu|\,.\label{auE:normZ}
\]
(b) For every $t\ge 0$ the mapping $\vecu\to Z_t(\omega,\vecu)$ defines a linear bounded operator $Z_t(\omega)$ on $\mathbb L^2(D)$.  In particular,
\begin{equation}\label{auE:sum}
Z_t(\omega,\vecu + \vecv)=
Z_t(\omega,\vecu) + Z_t(\omega,\vecv)\,.
\end{equation}
Moreover, for every $T>0$ there exists a constant $C_T>0$ such that 
\begin{equation}\label{sup}
\mE\sup_{t\le T}\left|Z_t(\vecu)\right|^2_{\mL^2(D)}\le C_T|\vecu|^2_{\mL^2(D)}\,.
\end{equation}
(c) For every $t\ge 0$ the operator $Z_t(\omega)$ is invertible and the inverse operator is the unique solution of the stochastic differential equation on $\mathbb L^2(D)$: 
\begin{equation}\label{auE:1inv}
Z_t^{-1}(\vecu) 
= \vecu-\sum_{i=1}^q\int_0^tZ_s^{-1}G_i(\vecu)\circ dW_i(s),\quad\vecu\in \mL^2(D)\,.
\end{equation}
Finally, 
\begin{equation}\label{auE:adj}
Z_t^{-1}(\omega)=Z_t^\star(\omega). 
\end{equation}
\end{lemma}
\begin{lemma}\label{lem: gradZ}
Assume that $\vecg_i\in C_b^{1+\alpha}\left(D,\mathbb R^3\right)$  for $i=1,\cdots,q$. Then, for every $\vecu\in\mH^1(D)$
$Z_t(\vecu)\in \mH^1$ $\mP$-a.s. Furthermore, the process $\xi_t(\vecu):=\nabla Z_t(\vecu)$, is the unique solution of the linear equation
\begin{equation*}
d\xi_t(\vecu)
=
\frac12\sum_{i=1}^q
\bigl(
G_i^2\xi_t(\vecu) + H_iZ_t(\vecu)
\bigr)\dt
+
\sum_{i=1}^q
\bigl(
G_i\xi_t(\vecu) + I_i Z_t(\vecu)
\bigr)\,dW_i(t),
\end{equation*}
with $\xi_0(\vecu) = \nabla\vecu$. \end{lemma}
\begin{lemma}\label{lem:auxi2}
For any $\vecu,\vecv\in\mL^2(D)$, there holds for all $t\geq 0$ and $\mP$-a.s.:
\begin{align}
Z_t(\vecu\times\vecv)
&=
Z_t(\vecu)\times Z_t(\vecv),\label{auE:cross}
\end{align}
\end{lemma}
\begin{proof}
Let $Z_t^{\vecu}:=Z_t(\vecu)$ and $Z_t^{\vecv}:=Z_t(\vecv)$ for all $t\geq 0$. 
We now prove~\eqref{auE:cross}; the property~\eqref{auE:sum} can be obtained in the same manner.
Using the It\^o formula for $Z_t^{\vecu}\times Z_t^{\vecv}$ and~\eqref{auE:2}, we obtain
\begin{align}
d(Z_t^{\vecu}\times Z_t^{\vecv})
&=
dZ_t^{\vecu}\times Z_t^{\vecv}
+ 
Z_t^{\vecu}\times dZ_t^{\vecv}
+
\sum_{i=1}^q
(Z_t^{\vecu}\times \vecg_i)\times (Z_t^{\vecv}\times \vecg_i) \dt\nn\\
&=
\sum_{i=1}^q
\big(
Z_t^{\vecu}\times (Z_t^{\vecv}\times\vecg_i)
-
Z_t^{\vecv}\times(Z_t^{\vecu}\times\vecg_i)  
\big)dW_i(t)\nn\\
&\quad+
\frac{1}{2}
\sum_{i=1}^q
\left(
Z_t^{\vecu}\times \big((Z_t^{\vecv}\times\vecg_i)\times\vecg_i\big)
-
Z_t^{\vecv}\times\big(Z_t^{\vecu}\times\vecg_i) \times\vecg_i\big) 
\right)\dt\label{auE:3}\\
&\quad+
\sum_{i=1}^q
(Z_t^{\vecu}\times \vecg_i)\times (Z_t^{\vecv}\times \vecg_i) \dt.\nn
\end{align}
Using an elementary identity 
\begin{equation}\label{equ:elemabc}
\veca\times(\vecb\times\vecc)
=
\inpro{\veca}{\vecc}\vecb
-
\inpro{\veca}{\vecb}\vecc\,,\quad \veca,\vecb,\vecc\in\R^3\,,
\end{equation}
we find that 
\begin{equation}\label{auE:4}
Z_t^{\vecu}\times (Z_t^{\vecv}\times\vecg_i)
-
Z_t^{\vecv}\times(Z_t^{\vecu}\times\vecg_i)  
=
(Z_t^{\vecu}\times Z_t^{\vecv})\times\vecg_i
\end{equation}
and
\begin{align}
Z_t^{\vecu}\times& \big((Z_t^{\vecv}\times\vecg_i)\times\vecg_i\big)
-
Z_t^{\vecv}\times\big(Z_t^{\vecu}\times\vecg_i) \times\vecg_i\big) \nn\\
&=
\inpro{Z_t^{\vecu}}{\vecg_i}(Z_t^{\vecv}\times\vecg_i)
-
\inpro{Z_t^{\vecv}}{\vecg_i}(Z_t^{\vecu}\times\vecg_i)
-
2(Z_t^{\vecu}\times\vecg_i) \times (Z_t^{\vecv}\times\vecg_i) \nn \\
&=
\big((Z_t^{\vecu}\times Z_t^{\vecv}) \times\vecg_i\big)\times\vecg_i
-
2(Z_t^{\vecu}\times\vecg_i) \times (Z_t^{\vecv}\times\vecg_i) .\label{auE:5}
\end{align}
Invoking~\eqref{auE:4} and~\eqref{auE:5}, equation~\eqref{auE:3} we obtain
\begin{equation*}
d(Z_t^{\vecu}\times Z_t^{\vecv})
=
\frac{1}{2}
\sum_{i=1}^q
\big((Z_t^{\vecu}\times Z_t^{\vecv}) \times\vecg_i\big)\times\vecg_i\dt
+
\sum_{i=1}^q
\big((Z_t^{\vecu}\times Z_t^{\vecv})\times\vecg_i\big)dW_i(t).
\end{equation*}
Therefore, the process $V_t:=Z_t^{\vecu}\times Z_t^{\vecv}$ is a solution of the following stochastic differential equation:
\begin{equation*}
\begin{cases}
&dV_t
=
\frac{1}{2}
\sum_{i=1}^q
\big(V_t\times\vecg_i\big)\times\vecg_i\dt
+
\sum_{i=1}^q
\big(V_t\times\vecg_i\big)dW_i(t)\\
&V_0 = \vecu\times\vecv.
\end{cases}
\end{equation*}
On the other hand, it follows from~\eqref{auE:2} that the process $Z_t(\vecu\times\vecv)$ satisfies the same equation. Hence, ~\eqref{auE:cross} follows from the uniqueness of solutions to~\eqref{auE:3}. 

\end{proof}
\begin{lemma}\label{lem: gradZ2}
For any $\vecu,\vecv\in\mH^1(D)$, there holds for all $t\geq 0$ and $\mP$-a.s.:
\begin{align}\label{eq: nabZ4}
\inpro{\nabla Z_t(\vecu)}{\nabla Z_t(\vecv)}_{\mL^2(D)}
&=
\inpro{\nabla \vecu}{\nabla \vecv}_{\mL^2(D)}
+
F(t,\vecu,\vecv),
\end{align}
with
\[
F(t,\vecu,\vecv):=
\sum_{i=1}^q
\int_0^t
F_{1,i}(s,\vecu,\vecv)\ds
+
\sum_{i=1}^q\int_0^t
F_{2,i}(s,\vecu,\vecv) \,dW_i(s)
\]
where 
\begin{align*}
F_{1,i}(t,\vecu,\vecv) 
&:=
\inpro{\nabla Z_t(\vecu)}{(\tfrac12H_i-G_iI_i)Z_t(\vecv)}_{\mL^2(D)}
-
\inpro{\nabla(\tfrac12H_i-G_iI_i)Z_t(\vecu)}{Z_t(\vecv)}_{\mL^2(D)}\nn\\
&\quad+
\inpro{I_i Z_t(\vecu)}{I_i Z_t(\vecv)}_{\mL^2(D)};
\end{align*}
and 
\begin{align*}
F_{2,i}(t,\vecu,\vecv)
:=
\inpro{\nabla Z_t(\vecu)}{I_i Z_t(\vecv)}_{\mL^2(D)}
-
\inpro{\nabla(I_i Z_t(\vecu))}{Z_t(\vecv)}_{\mL^2(D)}
\end{align*}
\end{lemma}
\begin{proof}
Let $\xi_t^{\vecu} :=\xi_t(\vecu)$ and $\xi_t^{\vecv} :=\xi_t(\vecv)$ for all $t\geq 0$. 
In addition, we consider a $C^{\infty}$ function $\phi: \bigl(\mL^2(D)\bigr)^2\goto \R$ 
defined by $\phi(\vecx,\vecy)= \inpro{\vecx}{\vecy}_{\mL^2(D)}$. By using the It\^o Lemma we obtain
\begin{align}
d\inpro{\xi_t^{\vecu}}{\xi_t^{\vecv}}_{\mL^2(D)}
&=
\inpro{d\xi_t^{\vecu}}{\xi_t^{\vecv}}_{\mL^2(D)}
+
\inpro{\xi_t^{\vecu}}{d\xi_t^{\vecv}}_{\mL^2(D)}
+
\inpro{d\xi_t^{\vecu}}{d\xi_t^{\vecv}}_{\mL^2(D)}\nn\\
&=
\sum_{i=1}^q
\bigg(
\tfrac12
\inpro{G_i^2\xi_t^{\vecu}+H_i Z_t^{\vecu}}{\xi_t^{\vecv}}_{\mL^2(D)}
+
\tfrac12
\inpro{\xi_t^{\vecu}}{G_i^2\xi_t^{\vecv}+H_i Z_t^{\vecv}}_{\mL^2(D)}\nn\\
&\quad\quad\quad+
\inpro{G_i\xi_t^{\vecu}+I_i Z_t^{\vecu}}{G_i\xi_t^{\vecv}+I_i Z_t^{\vecv}}_{\mL^2(D)}
\bigg)\dt\nn\\
&\quad+
\sum_{i=1}^q
\bigg(
\inpro{G_i\xi_t^{\vecu}+I_i Z_t^{\vecu}}{\xi_t^{\vecv}}_{\mL^2(D)}
+
\inpro{\xi_t^{\vecu}}{G_i\xi_t^{\vecv}+I_i Z_t^{\vecv}}_{\mL^2(D)}
\bigg) \,dW_i(t).\label{eq: nabZ1}
\end{align}
Using~\eqref{eq: nabZ2} and~\eqref{eq: nabZ3},
 we deduce from~\eqref{eq: nabZ1} that 
\begin{align*}
d\inpro{\xi_t^{\vecu}}{\xi_t^{\vecv}}_{\mL^2(D)}
&=
\sum_{i=1}^q
\bigg(
\tfrac12
\inpro{H_i Z_t^{\vecu}}{\xi_t^{\vecv}}_{\mL^2(D)}
+
\tfrac12
\inpro{\xi_t^{\vecu}}{H_i Z_t^{\vecv}}_{\mL^2(D)}
+
\inpro{G_i\xi_t^{\vecu}}{I_i Z_t^{\vecv}}_{\mL^2(D)}\nn\\
&\quad\quad\quad+
\inpro{I_i Z_t^{\vecu}}{G_i\xi_t^{\vecv}}_{\mL^2(D)}
+
\inpro{I_i Z_t^{\vecu}}{I_i Z_t^{\vecv}}_{\mL^2(D)}
\bigg)\dt\nn\\
&\quad+
\sum_{i=1}^q
\bigg(
\inpro{I_i Z_t^{\vecu}}{\xi_t^{\vecv}}_{\mL^2(D)}
+
\inpro{\xi_t^{\vecu}}{I_i Z_t^{\vecv}}_{\mL^2(D)}
\bigg) \,dW_i(t)\nn\\
&=
\sum_{i=1}^q
\bigg(
\inpro{(\tfrac12H_i-G_iI_i)Z_t^{\vecu}}{\xi_t^{\vecv}}_{\mL^2(D)}
+
\inpro{\xi_t^{\vecu}}{(\tfrac12H_i-G_iI_i)Z_t^{\vecv}}_{\mL^2(D)}\nn\\
&\quad\quad\quad+
\inpro{I_i Z_t^{\vecu}}{I_i Z_t^{\vecv}}_{\mL^2(D)}
\bigg)\dt\nn\\
&\quad+
\sum_{i=1}^q
\bigg(
\inpro{\xi_t^{\vecu}}{I_i Z_t^{\vecv}}_{\mL^2(D)}
+
\inpro{I_i Z_t^{\vecu}}{\xi_t^{\vecv}}_{\mL^2(D)}
\bigg) \,dW_i(t).
\end{align*}
Integrating by parts for the first and the last term in the right hand side of the  
above equation and noting the homogeneous Neumann boundary condition of $\vecg_i$, we obtain
\begin{align}\label{eq: nabZ5}
d\inpro{\xi_t^{\vecu}}{\xi_t^{\vecv}}_{\mL^2(D)}
&=
\sum_{i=1}^q
\bigg(
-
\inpro{\nabla(\tfrac12H_i-G_iI_i)Z_t^{\vecu}}{Z_t^{\vecv}}_{\mL^2(D)}
+
\inpro{\xi_t^{\vecu}}{(\tfrac12H_i-G_iI_i)Z_t^{\vecv}}_{\mL^2(D)}\nn\\
&\quad\quad\quad+
\inpro{I_i Z_t^{\vecu}}{I_i Z_t^{\vecv}}_{\mL^2(D)}
\bigg)\dt\nn\\
&\quad+
\sum_{i=1}^q
\bigg(
\inpro{\xi_t^{\vecu}}{I_i Z_t^{\vecv}}_{\mL^2(D)}
-
\inpro{\nabla(I_i Z_t^{\vecu})}{Z_t^{\vecv}}_{\mL^2(D)}
\bigg) \,dW_i(t).
\end{align}
Hence, the resutl follows from replacing $t$ by $s$ and intergrating~\eqref{eq: nabZ5} over $[0,t]$.
\end{proof}
\begin{remark}\label{rem: Fsym}
By using integration by parts and the homogeneous Neumann boundary conditions 
of $\vecg_i$ for $i=1,\cdots,q$ we obtain some symmetry properties of functions $F_{1,i}$, 
$F_{2,i}$ and $F$:
 for  any 
$\vecu,\vecv,\vecv_1,\vecv_2\in \mH^1(D)$, 
\[
F_{1,i}(t,\vecu,\vecv) = F_{1,i}(t,\vecv,\vecu);
\quad
F_{2,i}(t,\vecu,\vecv) = F_{2,i}(t,\vecv,\vecu);
\]
and hence, $F(t,\vecu,\vecv) = F(t,\vecv,\vecu)$. Furthermore, it follows from~\eqref{auE:sum} that
\[
F(t,\vecu,\vecv_1+\vecv_2) = F(t,\vecu,\vecv_1) + F(t,\vecu,\vecv_2).
\]
\end{remark}

The following lemmas state some important properties of $F$ used throughout this paper.
\begin{lemma}\label{lem: boundExF}
Assume that  
$\vecg_i\in\mW^{2,\infty}(D)$ for $i=1,\cdots,q$.
Then for any $\vecu,\vecv\in L^2(\Omega;\mH^1(D))$ there exists a constant $c$ depending on  $T$ and $\{\vecg_i\}_{i=1,\cdots,q}$ such that
\begin{align}
\mE\sup_{t\in[0,T]}\|\nabla Z_t(\vecu)\|_{\mL^2(D)}^2
\leq
c\mE\|\vecu\|_{\mH^1(D)}^2,\label{auE:normdeZ}
\end{align}
and for any $\epsilon >0$,
\begin{align}
\mE\sup_{s\in[0,t]}
\bigl|F(s,\vecu,\vecv)\bigr|
\leq 
c\epsilon 
\mE\|\vecu\|_{\mH^1(D)}^2
+
c\epsilon^{-1}
\mE\|\vecv\|_{\mL^2(D)}^2.\label{eq: boundExF}
\end{align}
\end{lemma}
\begin{proof}
It follows from~\eqref{eq: nabZ4} that
\begin{align}\label{eq: ExF6}
\mE \|\nabla Z_t(\vecu)\|_{\mL^2(D)}^2
&=
\mE \|\nabla\vecu\|_{\mL^2(D)}^2
+
\mE[ F(t,\vecu,\vecu)]\\
&\leq 
\mE \|\nabla\vecu\|_{\mL^2(D)}^2
+
\sum_{i=1}^q
\mE
\int_0^t
\bigl| F_{1,i}(\tau,\vecu,\vecu)\bigr|\,d\tau\nn\\
&\quad+
\mE
\bigl|
\sum_{i=1}^q\int_0^t
F_{2,i}(\tau,\vecu,\vecu) \,dW_i(\tau)
\bigl|.
\end{align}

For convenience, we next estimate $| F(\tau,\vecu,\vecv)|$, which is 
a slightly more general version of $| F(\tau,\vecu,\vecu)|$. 
By using the elementary inequality 
\begin{equation}\label{eq: eleCauchy}
 ab\leq \tfrac12\epsilon a^2 + \tfrac12\epsilon^{-1} b^2,  
\end{equation}
the assumption $\vecg_i\in\mW^{2,\infty}(D)$ and~\eqref{auE:normZ}, there holds 
\begin{align*}
\bigl| F_{1,i}(\tau,\vecu,\vecv)\bigr|
&\leq 
\bigl|
\inpro{\nabla Z_{\tau}(\vecu)}{(\tfrac12H_i-G_iI_i)Z_{\tau}(\vecv)}_{\mL^2(D)}\bigr|\nn\\
&\quad+
\bigl|
\inpro{\nabla(\tfrac12H_i-G_iI_i)Z_{\tau}(\vecu)}{Z_{\tau}(\vecv)}_{\mL^2(D)}\bigr|
+
\bigl|
\inpro{I_i Z_{\tau}(\vecu)}{I_i Z_{\tau}(\vecv)}_{\mL^2(D)}\bigr|\nn\\
&\leq 
c\bigl(\epsilon\|\nabla Z_{\tau}(\vecu)\|^2_{\mL^2(D)}
+
\epsilon\|Z_{\tau}(\vecu)\|^2_{\mL^2(D)}
+
\epsilon^{-1}\|Z_{\tau}(\vecv)\|^2_{\mL^2(D)}\bigr)\nn\\
&\leq 
c\bigl(\epsilon\|\nabla Z_{\tau}(\vecu)\|^2_{\mL^2(D)}
+
\epsilon\|\vecu\|^2_{\mL^2(D)}
+
\epsilon^{-1}\|\vecv\|^2_{\mL^2(D)}\bigr).
\end{align*}
This implies that 
\begin{align}\label{eq: ExF2}
\mE
\int_0^t
\bigl| F_{1,i}(\tau,\vecu,\vecv)\bigr|\,d\tau
\leq 
c\epsilon t\mE\|\vecu\|^2_{\mL^2(D)}
+
c\epsilon^{-1}t\mE\|\vecv\|^2_{\mL^2(D)}
+
c\epsilon\mE
\int_0^t
\|\nabla Z_{\tau}(\vecu)\|^2_{\mL^2(D)}\,d\tau.
\end{align}
Then, by using the Burkholder-Davis-Gundy inequality, H\"older inequality,~\eqref{auE:normZ} and~\eqref{eq: eleCauchy}, 
we estimate 
\begin{align}\label{eq: exF4}
\mE\sup_{s\in[0,t]}
\bigl|
&\sum_{i=1}^q\int_0^s
F_{2,i}(\tau,\vecu,\vecv) \,dW_i(\tau)
\bigl|
\leq 
c
\mE
\bigl|
\sum_{i=1}^q\int_0^t
\bigl(F_{2,i}(\tau,\vecu,\vecv)\bigr)^2 \,d\tau
\bigl|^{1/2}\nn\\
&\leq 
c
\mE
\bigl|
\int_0^t
\bigl(
\|\nabla Z_{\tau}(\vecu)\|_{\mL^2(D)} 
\|Z_{\tau}(\vecv)\|_{\mL^2(D)}
+
\|Z_{\tau}(\vecu)\|_{\mL^2(D)}
\|Z_{\tau}(\vecv)\|_{\mL^2(D)}
\bigr)^2 \,d\tau
\bigl|^{1/2}\nn\\
&\leq
c
\mE
\bigl|
\int_0^t
\bigl(
\|\nabla Z_{\tau}(\vecu)\|^2_{\mL^2(D)} 
\|\vecv\|^2_{\mL^2(D)}
+
\|\vecu\|^2_{\mL^2(D)}
\|\vecv\|^2_{\mL^2(D)}
\bigr) \,d\tau
\bigl|^{1/2}\nn\\
&\leq 
c
\mE\left[
\|\vecv\|_{\mL^2(D)}
\bigl(
\int_0^t
\|\nabla Z_{\tau}(\vecu)\|^2_{\mL^2(D)} \,d\tau
\bigr) ^{1/2}
\right]
+
ct^{1/2}\mE\left[
\|\vecu\|_{\mL^2(D)}
\|\vecv\|_{\mL^2(D)}
\right]\\
&\leq 
c\epsilon t^{1/2}
\mE\|\vecu\|^2_{\mL^2(D)}
+
c\epsilon^{-1} (t^{1/2}+1)\mE
\|\vecv\|^2_{\mL^2(D)}
+
c\epsilon \mE
\int_0^t
\|\nabla Z_{\tau}(\vecu)\|^2_{\mL^2(D)}\,d\tau.\label{eq: exF5}
\end{align}

We use~\eqref{eq: ExF2} and~\eqref{eq: exF5} 
with $\vecv=\vecu$ and $\epsilon = 1$ together with~\eqref{eq: ExF6} to deduce
\begin{align*}
\mE \|\nabla Z_t(\vecu)\|_{\mL^2(D)}^2
&\leq 
c\mE \|\vecu\|_{\mH^1(D)}^2
+
c\mE
\int_0^t
\|\nabla Z_{\tau}(\vecu)\|^2_{\mL^2(D)}\,d\tau.
\end{align*}
Hence, the result~\eqref{auE:normdeZ} follows immediately by using Gronwall's inequality.

To prove~\eqref{eq: boundExF} we  note that 
\begin{align}\label{eq: ExF1}
\mE\sup_{s\in[0,t]}
\bigl|F(s,\vecu,\vecv)\bigr|
&\leq 
\sum_{i=1}^q
\mE
\int_0^t
\bigl| F_{1,i}(\tau,\vecu,\vecv)\bigr|\,d\tau
\nn\\
&\quad+
\mE\sup_{s\in[0,t]}
\bigl|
\sum_{i=1}^q\int_0^s
F_{2,i}(\tau,\vecu,\vecv) \,dW_i(\tau)
\bigl|.
\end{align}
Hence, it follows from~\eqref{eq: ExF2},~\eqref{eq: exF5} and~\eqref{auE:normdeZ} that 
\begin{align*}
\mE\sup_{s\in[0,t]}
\bigl|
\sum_{i=1}^q\int_0^s
F_{2,i}(\tau,\vecu,\vecv) \,dW_i(\tau)
\bigl|
&\leq 
c\epsilon \mE\|\vecu\|^2_{\mL^2(D)}
+
c\epsilon^{-1}\mE\|\vecv\|^2_{\mL^2(D)}\nn\\
&\quad+
c\epsilon\mE
\int_0^t
\|\nabla Z_{\tau}(\vecu)\|^2_{\mL^2(D)}\,d\tau\nn\\
&\leq 
c\bigl(
\epsilon \mE\|\vecu\|^2_{\mH^1(D)}
+
\epsilon^{-1} 
\mE\|\vecv\|^2_{\mL^2(D)}
\bigr),
\end{align*}
which completed the proof of the lemma.
\end{proof}
\begin{lemma}\label{lem: Fbound9}
For any $\vecu\in L^2\bigl(\Omega;\mL^2(D)\bigr)$, $\vecv\in L^2\bigl(\Omega;\mH^1(D)\bigr)$ and 
$0\leq  s\leq T$, there exists a constant $c$ depending on 
$\{\vecg_i\}_{i=1,\cdots,q}$ such that 
\begin{align*}
\mE
|F( s,\vecu ,\vecv )| 
&\leq 
c s \bigl(\mE[
\|\vecu \|_{\mL^2(D)}^2]\bigr)^{1/2}
\bigl(\mE
[\|\nabla\vecv \|_{\mL^2(D)}^2]\bigr)^{1/2}\nn\\
&\quad+
c( s^{1/2}+ s)\bigl(\mE[
\|\vecu \|_{\mL^2(D)}^2]\bigr)^{1/2}
\bigl(\mE[\|\vecv \|_{\mL^2(D)}^2
]\bigr)^{1/2}.
\end{align*}
\end{lemma}
\begin{proof}
From the definition of function $F$ in Lemma~\ref{lem: gradZ2} and the triangle inequality, there holds
\begin{align}\label{eq: Fbound1}
\mE
|F( s,\vecu ,\vecv )|
\leq 
\sum_{i=1}^q
\mE
\int_0^{ s}
\bigl| F_{1,i}(\tau,\vecu ,\vecv )\bigr|\,d\tau
+
\mE
\bigl|
\sum_{i=1}^q\int_0^{ s}
F_{2,i}(\tau,\vecu ,\vecv ) \,dW_i(\tau)
\bigl|.
\end{align}
From Remark~\ref{rem: Fsym}, we note that
\begin{align*}
F_{2,i}(\tau,\vecu ,\vecv )
&=
F_{2,i}(\tau,\vecv ,\vecu ),
\end{align*}
and therefore, by using~\eqref{eq: exF4}, the last term of~\eqref{eq: Fbound1} 
can be estimated as follows:
\begin{align}\label{eq: Fbound4}
\mE
\bigl|
\sum_{i=1}^q\int_0^{ s}
F_{2,i}(\tau,\vecu ,\vecv ) \,dW_i(\tau)
\bigl|
&=
\mE
\bigl|
\sum_{i=1}^q\int_0^{ s}
F_{2,i}(\tau,\vecv ,\vecu ) \,dW_i(\tau)
\bigl|\nn\\
&\leq
c
\mE\left[
\|\vecu \|_{\mL^2(D)}
\bigl(
\int_0^{ s}
\|\nabla Z_{\tau}(\vecv )\|^2_{\mL^2(D)} \,d\tau
\bigr) ^{1/2}
\right]\nn\\
&\quad+
c s^{1/2}\mE\left[
\|\vecu \|_{\mL^2(D)}
\|\vecv \|_{\mL^2(D)}
\right]. 
\end{align}
We now estimate $\bigl| F_{1,i}(\tau,\vecu ,\vecv )\bigr|$ by 
integrating by parts and then using 
H\"older's inequality, the assumption $\vecg_i\in \mW^{2,\infty}(D)$ and~\eqref{auE:normZ} as follows:
\begin{align*}
\bigl| F_{1,i}(\tau,\vecu ,\vecv )\bigr|
&=
\bigl|
-\inpro{Z_{\tau}(\vecu )}{\nabla\bigl((\tfrac12H_i-G_iI_i)Z_{\tau}(\vecv )\bigr)}_{\mL^2(D)}\nn\\
&\quad+
\inpro{(\tfrac12H_i-G_iI_i)Z_{\tau}(\vecu )}{\nabla Z_{\tau}(\vecv )}_{\mL^2(D)}\nn\\
&\quad+
\inpro{I_i Z_{\tau}(\vecu )}{I_i Z_{\tau}(\vecv )}_{\mL^2(D)}
\bigr|\nn\\
&\leq 
c
\|\vecu \|_{\mL^2(D)}
\bigl(
\|\nabla Z_{\tau}\vecv \|_{\mL^2(D)}
+
\|\vecv \|_{\mL^2(D)}
\bigr),
\end{align*}
and therefore,
\begin{align}\label{eq: Fbound5}
\mE
\int_0^{ s}
\bigl| F_{1,i}(\tau,\vecu ,\vecv )\bigr|\,d\tau
&\leq
c \mE\left[
\|\vecu \|_{\mL^2(D)}
\bigl(
\int_0^{ s}\|\nabla Z_{\tau}\vecv \|_{\mL^2(D)}\,d\tau \bigr)
\right]\nn\\
&\quad+
c s\mE\left[
\|\vecu \|_{\mL^2(D)}
\|\vecv \|_{\mL^2(D)}
\right].
\end{align}
Hence, by using 
H\"older inequality we obtain  from~\eqref{eq: Fbound1}--\eqref{eq: Fbound5} that there holds:
\begin{align}\label{eq: Fbound8}
\mE
|F( s,\vecu ,\vecv )|
&\leq 
c \mE\left[
\|\vecu \|_{\mL^2(D)}
\bigl(
\int_0^{ s}\|\nabla Z_{\tau}\vecv \|_{\mL^2(D)}\,d\tau \bigr)
\right]\nn\\
&\quad+
c( s^{1/2}+ s)\mE\left[
\|\vecu \|_{\mL^2(D)}
\|\vecv \|_{\mL^2(D)}
\right]\nn\\
&\leq 
c \bigl(\mE[
\|\vecu \|_{\mL^2(D)}^2]\bigr)^{1/2}
\bigl(\mE
\bigl[(
\int_0^{ s}\|\nabla Z_{\tau}\vecv \|_{\mL^2(D)}\,d\tau)^2 \bigr]
\bigr)^{1/2}\nn\\
&\quad+
c( s^{1/2}+ s)\bigl(\mE[
\|\vecu \|_{\mL^2(D)}^2]\bigr)^{1/2}
\bigl(\mE[\|\vecv \|_{\mL^2(D)}^2
]\bigr)^{1/2}.
\end{align}
Via the Minkowski inequality and~\eqref{auE:normdeZ}, we observe that
\begin{align}\label{eq: Fbound3}
\bigl(\mE
\bigl[(
\int_0^{ s}\|\nabla Z_{\tau}\vecv \|_{\mL^2(D)}\,d\tau)^2 \bigr]
\bigr)^{1/2}
\leq 
\int_0^{ s}
\bigl(\mE
[\|\nabla Z_{\tau}\vecv \|_{\mL^2(D)}^2]\bigr)^{1/2}\,d\tau
\leq 
c s
\bigl(\mE
[\|\nabla\vecv \|_{\mL^2(D)}^2]\bigr)^{1/2}.
\end{align}
The required result follows from~\eqref{eq: Fbound8} and~\eqref{eq: Fbound3},
which completes the proof of this lemma.
\end{proof}
\section{Equivalence of weak solutions}\label{sec:equ eqn}
In this section we use the process $(Z_t)_{t\geq 0}$ defined in the
preceding section to define a new process $\vecm$ from
$\vecM$.
Let
\begin{equation}\label{equ:vecm}
\vecm(t,\vecx) = Z_t^{-1}\vecM(t,\vecx) \quad\forall t
\ge0, \ a.e. \vecx\in D.
\end{equation}
We will show that this new variable $\vecm$ is differentiable with
respect to $t$.

In the next lemma, we introduce the equation satisfied by $\vecm$ so that $\vecM$ is a solution to~\eqref{E:1.1} in the sense of~\eqref{wE:1.1}.
\begin{lemma}\label{lem:4.2}
If  
 $\vecm(\cdot,\omega)\in H^1(0,T;\mL^2(D))\cap L^2(0,T;\mH^1(D))$, 
for $\mP$-a.s. $\omega\in\Omega$, satisfies
\[
\snorm{\vecm(t,\vecx)}{} = 1
\quad\forall t\ge0, \ a.e.\  \vecx\in D,\, \mP-\text{a.s.},
\]
and for any $\vecpsi\in L^2(0,T;\mW^{1,\infty}(D))$
\begin{align}\label{InE:14}
\inpro{\partial_t\vecm}{\vecpsi}_{\mL^2(D_T)}
&+
\lambda_1\int_0^T
\inpro{Z_s\vecm\times\nabla Z_s\vecm}{\nabla Z_s\vecpsi}_{\mL^2(D)}\ds\nn\\
&+
\lambda_2\int_0^T
\inpro{Z_s\vecm\times\nabla Z_s\vecm}{\nabla Z_s (\vecm\times\vecpsi)}_{\mL^2(D)}\ds
=0,
\quad \mP\text{-a.s.}.
\end{align}
Then $\vecM = Z_t\vecm$ satisfies \eqref{wE:1.1} $\mP$-a.s..
\end{lemma}
\begin{proof} 
Using It\^o's formula for $\vecM= Z_t\vecm$,  we deduce
\begin{align*}
\vecM(t)
&=
\vecM(0)
+
\sum_{i=1}^q
\int_0^t Z\vecm\times\vecg_i \circ dW_i(s)
+ \int_0^t Z(\partial_t\vecm) \ds\\
&=
\vecM(0)
+
\sum_{i=1}^q
\int_0^t \vecM\times\vecg_i \circ dW_i(s)
+ \int_0^t Z_s(\partial_t\vecm)\ds.
\end{align*}
Multiplying both sides by a test function
$\vecpsi\in\C_0^{\infty}(D)$ and integrating over $D$ we
obtain
\begin{align}\label{equ:wIto}
\inpro{\vecM(t)}{\vecpsi}_{\mL^2(D)}
&=
\inpro{\vecM(0)}{\vecpsi}_{\mL^2(D)}
+
\sum_{i=1}^q
\int_0^t
\inpro{\vecM\times\vecg_i}{\vecpsi}_{\mL^2(D)}\circ dW_i(s)\nn\\
&\quad+
\int_0^t
\inpro{Z_s(\partial_t\vecm)}{\vecpsi}_{\mL^2(D)}\ds\nn\\
&=
\inpro{\vecM(0)}{\vecpsi}_{\mL^2(D)}
+
\sum_{i=1}^q
\int_0^t
\inpro{\vecM\times\vecg_i}{\vecpsi}_{\mL^2(D)}\circ dW_i(s)\nn\\
&\quad+
\int_0^t
\inpro{\partial_t\vecm}{Z_s^{-1}\vecpsi}_{\mL^2(D)}\ds,
\end{align}
where in the last step we used~\eqref{auE:adj}.
On the other hand, it follows from~\eqref{InE:14} that,
for all $\vecxi\in  L^2(0,t;\mW^{1,\infty}(D))$, there holds:
\begin{align}\label{equ:dm xi}
\int_0^t
\inpro{\partial_t\vecm}{\vecxi}_{\mL^2(D)}\ds
&=
-\lambda_1
\int_0^t
\inpro{Z_s\vecm\times\nabla Z_s\vecm}{\nabla Z_s\vecxi}_{\mL^2(D)}\ds\nn\\
&\quad-\lambda_2
\int_0^t
\inpro{Z_s\vecm\times\nabla Z_s\vecm}{\nabla Z_s(\vecm\times\vecxi)}_{\mL^2(D)}\ds.
\end{align}
Using~\eqref{equ:dm xi} with $\vecxi=Z_s^{-1}\vecpsi$ for the last term on the right hand side of~\eqref{equ:wIto} we deduce  
\begin{align*}
\inpro{\vecM(t)}{\vecpsi}_{\mL^2(D)}
&=
\inpro{\vecM(0)}{\vecpsi}_{\mL^2(D)}
+
\sum_{i=1}^q
\int_0^t
\inpro{\vecM\times\vecg_i}{\vecpsi}_{\mL^2(D)}\circ dW_i(s)\nn\\
&\quad-\lambda_1
\int_0^t
\inpro{\vecM\times\nabla\vecM}{\nabla\vecpsi\big)}_{\mL^2(D)}\ds \\
&\quad
-\lambda_2
\int_0^t
\inpro{\vecM\times\nabla\vecM}{\nabla Z_s(\vecm\times Z_s^{-1}\vecpsi)}_{\mL^2(D)}\ds.
\end{align*}
It follows from~\eqref{auE:cross} that
\begin{align*}
\inpro{\vecM(t)}{\vecpsi}_{\mL^2(D)}
&=
\inpro{\vecM(0)}{\vecpsi}_{\mL^2(D)}
+
\sum_{i=1}^q
\int_0^t
\inpro{\vecM\times\vecg_i}{\vecpsi}_{\mL^2(D)}\circ dW_i(s)\nn\\
&\quad-\lambda_1
\int_0^t
\inpro{\vecM\times\nabla\vecM}{\nabla\vecpsi\big)}_{\mL^2(D)}\ds \\
&\quad
-\lambda_2
\int_0^t
\inpro{\vecM\times\nabla\vecM}{\nabla (\vecM\times\vecpsi)}_{\mL^2(D)}\ds,
\end{align*}
which complete the proof.
\end{proof}
The following lemma shows that the constraint on $|\vecm |$ is inherited by $|\vecM |$.
\begin{lemma}\label{lem:m 1}
The process $\vecM$ satisfies
\[
\snorm{\vecM(t,\vecx)}{} = 1
\quad\forall t\ge0, \ a.e.\, \vecx\in D,\, \mP-\text{a.s.}
\]
if and only if $\vecm$ defined in~\eqref{equ:vecm} satisfies
\[
\snorm{\vecm(t,\vecx)}{} = 1
\quad\forall t\ge0, \ a.e. \vecx\in D,\, \mP-\text{a.s.}.
\]
\end{lemma}
\begin{proof}
The proof follows by using~\eqref{auE:adj}: 
\begin{align*}
|\vecm|^2=\inpro{\vecm}{\vecm}
=\inpro{Z^{-1}_t\vecM}{Z^{-1}_t\vecM}
=\inpro{\vecM}{Z_tZ^{-1}_t\vecM}
=\inpro{\vecM}{\vecM}=|\vecM|^2.
\end{align*}
\end{proof}
In the next lemma we provide a relationship between equation~\eqref{InE:14} and its Gilbert form.
\begin{lemma}\label{lem:4.1}
Let $\vecm(\cdot,\omega)\in H^1(0,T;\mL^2(D))\cap L^2(0,T;\mH^1(D))$ for $\mP$-a.s. $\omega\in\Omega$ satisfy
\begin{equation}\label{equ:m 1}
|\vecm(t,\vecx)| = 1, \quad t\in(0,T), \ \vecx\in D,
\end{equation}
and
\begin{align}\label{InE:13}
\lambda_1\inpro{\partial_t\vecm}{\vecvarphi}_{\mL^2(D_T)}
+
\lambda_2\inpro{\vecm\times\partial_t\vecm}{\vecvarphi}_{\mL^2(D_T)}
=
\mu\int_0^T \inpro{\nabla Z_s\vecm}{\nabla Z_s(\vecm\times\vecvarphi)}_{\mL^2(D)}\ds,
\end{align}
where $\mu=\lambda_1^2+\lambda_2^2$. Then $\vecm$ satisfies~\eqref{InE:14}.
\end{lemma}
\begin{proof}
For each $\vecpsi\in L^2(0,T;\mW^{1,\infty}(D))$, using Lemma~\ref{lem:4.0}
in the Appendix, there exists
$\vecvarphi\in L^2(0,T;\mH^1(D))$ such that
\begin{equation}\label{Equ:phi}
\lambda_1{\vecvarphi}
+
\lambda_2{\vecvarphi}\times\vecm
=\vecpsi.
\end{equation}
We can write~\eqref{InE:13} as
\begin{align}\label{InE:15}
\quad\inpro{\partial_t\vecm}
{\lambda_1{\vecvarphi}+\lambda_2{\vecvarphi}\times\vecm}_{\mL^2(D_T)}
&+
\lambda_1\int_0^T
\inpro{Z_s\vecm\times\nabla Z_s\vecm}
{\nabla Z_s(\lambda_1{\vecvarphi})}_{\mL^2(D)}\ds\nn\\
&+\lambda_2\int_0^T
\inpro{\nabla Z_s\vecm}
{\nabla Z_s(\lambda_2{\vecvarphi}\times\vecm)}_{\mL^2(D)}\ds
=0.
\end{align}
From ~\eqref{equ:m 1} and~\eqref{auE:normZ}, we obtain that 
\begin{equation}\label{equ:Zm1}
|Z_t\vecm(t,\vecx)|=1,\quad \forall t\in (0,T)\,\text{and }\vecx\in D.
\end{equation}
On the other hand, by using~\eqref{equ:Zm1},~\eqref{equ:elemabc} and a standard identity 
\begin{equation}\label{equ:elemabc2}
\inpro{\veca}{\vecb\times\vecc}
=
\inpro{\vecb}{\vecc\times\veca}
=
\inpro{\vecc}{\veca\times\vecb},\quad\text{for all $\veca,\vecb,\vecc\in\R^3$},
\end{equation}
we obtain
\begin{align}\label{InE:16}
\lambda_1\int_0^T
&\inpro{Z_s\vecm\times\nabla Z_s\vecm}
{\nabla Z_s(\lambda_2{\vecvarphi}\times\vecm)}_{\mL^2(D)}\ds
+\lambda_2\int_0^T
\inpro{\nabla Z_s\vecm}
{\nabla Z_s(\lambda_1{\vecvarphi})}_{\mL^2(D)}\ds\nn\\
&-\lambda_2\int_0^T
\inpro{|\nabla Z_s\vecm|^2 Z_s\vecm}
{Z_s(\lambda_1{\vecvarphi})}_{\mL^2(D)}\ds=0.
\end{align}
Moreover, we have
\begin{equation}\label{InE:17}
-\lambda_2\int_0^T
\inpro{|\nabla Z_s\vecm|^2 Z_s\vecm}
{\lambda_2 Z_s\vecvarphi\times Z_s\vecm}_{\mL^2(D)}\ds=0.
\end{equation}
Summing \eqref{InE:15},~\eqref{InE:16} and~\eqref{InE:17} gives
\begin{align*}
\inpro{\partial_t\vecm}
{\lambda_1{\vecvarphi}+\lambda_2{\vecvarphi}\times\vecm}_{\mL^2(D_T)}
&+
\lambda_1\int_0^T
\inpro{Z_s\vecm\times\nabla Z_s\vecm}
{\nabla Z_s(\lambda_1{\vecvarphi}+\lambda_2{\vecvarphi}\times\vecm)}_{\mL^2(D)}\ds\nn\\
&+
\lambda_2\int_0^T
\inpro{\nabla Z_s\vecm}
{\nabla Z_s(\lambda_1{\vecvarphi}+\lambda_2{\vecvarphi}\times\vecm)}_{\mL^2(D)}\ds\nn\\
&-
\lambda_2\int_0^T
\inpro{|\nabla Z_s\vecm|^2 Z_s\vecm}
{Z_s(\lambda_1{\vecvarphi}+\lambda_2{\vecvarphi}\times\vecm)}_{\mL^2(D)}\ds
=0
\end{align*}
The desired equation \eqref{InE:14} follows by noting
\eqref{Equ:phi} and using~\eqref{equ:elemabc},~\eqref{equ:elemabc2} and~\eqref{equ:Zm1}.
\end{proof}

\begin{remark}\label{rem:LLL}
By using~\eqref{equ:elemabc2} and~\eqref{equ:m 1} we can rewrite~\eqref{InE:13}
as
\begin{align}\label{E:1.3a}
\lambda_1\inpro{\vecm\times\partial_t\vecm}{\vecw}_{\mL^2(D_T)}
-
\lambda_2\inpro{\partial_t\vecm}{\vecw}_{\mL^2(D_T)} 
=
\mu\int_0^T \inpro{\nabla Z_s\vecm}{\nabla Z_s\vecw}_{\mL^2(D)}\ds,
\end{align}
or equivalently, thanks to Lemma~\ref{lem: gradZ2},
\begin{align}\label{E:1.3ab}
\lambda_1\inpro{\vecm\times\partial_t\vecm}{\vecw}_{\mL^2(D_T)}
-
\lambda_2\inpro{\partial_t\vecm}{\vecw}_{\mL^2(D_T)} 
=
&\mu \inpro{\nabla\vecm}{\nabla\vecw}_{\mL^2(D_T)}\nn\\
&+
\mu \int_0^T F(t,\vecm(t,\cdot),\vecw(t,\cdot))\dt,
\end{align}
where $\vecw=\vecm\times\vecvarphi$ for
$\vecvarphi\in L^2(0,T;\mH^1(D))$.  
We note in particular that
$\vecw\cdot\vecm=0$. This property will be exploited later
in the design of the finite element scheme.
\end{remark}

We state the following lemma as a consequence of Lemmas~\ref{lem:4.1},~\ref{lem:m 1} and~\ref{lem:4.2}.
\begin{lemma}\label{lem:equi}
Let $\vecm(\cdot,\omega)\in H^1(0,T;\mL^2(D))\cap L^2(0,T;\mH^1(D))$ for $\mP$-a.s. $\omega\in\Omega$. 
If $\vecm$ is a solution of~\eqref{equ:m 1}--~\eqref{InE:13}, 
 then $\vecM = Z_t \vecm$ is a weak martingale solution of~\eqref{E:1.1} in the sense of Definition~\ref{def:wea sol}.
\end{lemma}
\begin{proof}
By using Lemmas~\ref{lem:4.2},~\ref{lem:m 1} and~\ref{lem:4.1} together with the 
imbedding $H^1(0,T;\mL^2(D))\hookrightarrow C(0,T;\mH^{-1}(D)$, we deduce that $\vecM$  satisfies 
$(1)$, $(2)$, $(3)$, $(4)$ in Definition~\ref{def:wea sol}, which completes the proof.
\end{proof}

Thanks to the above lemma, we now can now restrict our attention to solving equation~\eqref{InE:13} rather than ~\eqref{wE:1.1}.
\section{The finite element scheme}\label{sec:fin ele}
In this section we design a finite element scheme to find
approximate solutions to~\eqref{InE:13}. In the next section, we
prove that the
finite element solutions converge to a solution
of~\eqref{InE:13}. Then, thanks to Lemma~\ref{lem:equi},
we obtain a weak solution of~\eqref{wE:1.1}.

Let $\mT_h$ be a regular tetrahedrization of the domain
$D$ into tetrahedra of maximal mesh-size $h$.
We denote by $\cN_h := \{\vecx_1,\ldots,\vecx_N\}$ the set
of vertices and
introduce the finite-element space
$\mV_h\subset\mH^1(D)$, which is the space of all
continuous piecewise linear functions on $\mT_h$. A
basis for $\mV_h$ can be chosen to be $\{\phi_n\vecxi_1,\phi_n\vecxi_2,\phi_n\vecxi_3\}_{1\leq
n\leq N}$, where
 $\{\vecxi_i\}_{i=1,\cdots,3}$ is the canonical basis for $\R^3$ and $\phi_n(\vecx_m)=\delta_{n,m}.$  Here $\delta_{n,m}$
denotes
the Kronecker delta symbol.
The interpolation operator from
$\C^0(D)$ onto  $\mV_h$, denoted by $I_{\mV_h}$, is defined by
\[
I_{\mV_h}(\vecv)=\sum_{n=1}^N \vecv(\vecx_n)\phi_n(\vecx)
\quad\forall \vecv\in \mathbb C^0(D,\mathbb R^3) .
\]

Before introducing the finite element scheme,
 we state the following result proved by
Bartels~\cite{Bart05}, which will be used in the subsequent analysis.
\begin{lemma}\label{lem:bar}
Assume that
\begin{equation}\label{E:CondTe}
\int_{D} \nabla\phi_i\cdot\nabla\phi_j\dvx \leq 0
\quad\text{for all}\quad i,j \in \{1,2,\cdots,J\}\text{ and
} i\not= j .
\end{equation}
Then for all $\vecu\in\mV_h$ satisfying
$|\vecu(\vecx_l)|\geq 1$, $ l=1,2,\cdots,J$, there holds
\begin{equation}\label{E:InE}
\int_{D}\left|\nabla
I_{\mV_h}\left(\frac{\vecu}{|\vecu|}\right)\right|^2\dvx
\leq
\int_{D}|\nabla\vecu|^2\dvx.
\end{equation}
\end{lemma}
\noindent
When $d=2$, we note that condition~\eqref{E:CondTe} holds for Delaunay
triangulation. 
Roughly speaking, a Delaunay triangulation is one in which no vertex is contained inside the perimeter of any triangle. When $d=3$, condition~\eqref{E:CondTe} holds if all dihedral angles of
the tetrahedra in $\mT_h|_{\mL^2(D)}$ are less than or equal
to $\pi/2$; see~\cite{Bart05}.
In what follows we assume that~\eqref{E:CondTe} holds.

To discretize the equation~\eqref{InE:13}, we fix a positive integer $J$, choose the time step
$k$ to be $k=T/J$ and define $t_j=jk$, $j=0,\cdots,J$. For
$j=1,2,\ldots,J$, the solution $\vecm (t_j,\cdot)$
is approximated by $\vecm^{(j)}_h\in\mV_h$, which is
computed as follows.

Since
\[
\vecm_t(t_j,\cdot)
\approx
\frac{\vecm(t_{j+1},\cdot)-\vecm(t_j,\cdot)}{k}
\approx
\frac{\vecm_h^{(j+1)}-\vecm_h^{(j)}}{k},
\]
we can define $\vecm_h^{(j+1)}$ from $\vecm_h^{(j)}$ by
\begin{equation}\label{equ:mjp1}
\vecm_h^{(j+1)}
=
\vecm_h^{(j)} + k \vecv_h^{(j)},
\end{equation}
where $\vecv_h^{(j)}$ is an approximation of
$\vecm_t(t_j,\cdot)$. Hence, it suffices to propose a scheme
to compute $\vecv_h^{(j)}$.

Motivated by the property $\partial_t\vecm\cdot\vecm=0$, 
we
will find $\vecv_h^{(j)}$ in the space $\mW^{(j)}_h$ defined
by
\begin{equation}\label{equ:Whj}
 \mW_h^{(j)}
 :=
 \left\{\vecw\in \mV_h \mid
 \vecw(\vecx_n)\cdot\vecm_h^{(j)}(\vecx_n)=0,
 \ n = 1,\ldots,N \right\}.
\end{equation}
Given $\vecm_h^{(j)}\in\mV_h$,
we use~\eqref{E:1.3ab} to define $\vecv_h^{(j)}$ instead
of using~\eqref{InE:13} so that the same test and trial
functions can be used (see Remark~\ref{rem:LLL}).
Hence, we define by $\vecv_h^{(j)}\in\mW_h^{(j)}$ satisfying the following equation 
\begin{align}\label{E:1.5}
\lambda_1
\inpro{\vecm_h^{(j)}\times\vecv_h^{(j)}}{\vecw_h^{(j)}}_{\mL^2(D)}
-
\lambda_2
\inpro{\vecv_h^{(j)}}{\vecw_h^{(j)}}_{\mL^2(D)}
&=
\mu
\inpro{\nabla(\vecm_h^{(j)}+k\theta\vecv_h^{(j)})}{\nabla \vecw_h^{(j)}}_{\mL^2(D)}\nn\\
&\quad+
\mu F(t_j,\vecm_h^{(j)},\vecw_h^{(j)})\quad \mP\text{-a.s.}.
\end{align}

We summarise the algorithm as follows.
\bigskip
\begin{algorithm}\label{Algo:1}
\mbox{}
\begin{description}
\item[Step 1]
Set $j=0$.
Choose $\vecm^{(0)}_h=I_{\mV_h}\vecm_0$.
\item[Step 2]
Find $\vecv_h^{(j)}\in \mW_h^{(j)}$
satisfying~\eqref{E:1.5}.\label{A:2}
\item[Step 3] \label{A:4}
Define
\begin{equation*}
\vecm_h^{(j+1)}(\vecx)
:=
\sum_{n=1}^N
\frac{\vecm_h^{(j)}(\vecx_n)+k\vecv_h^{(j)}(\vecx_n)}
{\left|\vecm_h^{(j)}(\vecx_n)+k\vecv_h^{(j)}(\vecx_n)\right|}
\phi_n(\vecx).
\end{equation*}
\item[Step 4]
Set $j=j+1$, and return to Step 2 if $j<J$. Stop if
$j=J$.
\end{description}
\end{algorithm}
Since $\left|\vecm_h^{(0)}(x_n)\right|=1$ and
$\vecv_h^{(j)}(x_n)\cdot\vecm_h^{(j)}(x_n)=0$ for all
$n=1,\ldots,N$ and $j=0,\ldots,J$, we obtain (by induction)
\begin{equation}\label{equ:mhj 1}
\left |\vecm_h^{(j)}(x_n)+k\vecv_h^{(j)}(x_n)\right| \ge 1
\quad\text{and}\quad
\left |\vecm_h^{(j)}(x_n)\right |=1,
\quad j = 0,\ldots,J.
\end{equation}
In particular, ~\eqref{equ:mhj 1} shows that
the algorithm is well defined.

We finish this section by proving the following three lemmas
concerning some properties of
$\vecm_h^{(j)}$ and $R_{h,k}$.
\begin{lemma}\label{lem:mhj}
For any $j=0,\ldots,J$,
\[
\norm{\vecm_h^{(j)}}{\mL^{\infty}(D)} \le 1
\quad\text{and}\quad
\norm{\vecm_h^{(j)}}{\mL^2(D)} \le |D|,
\]
where $|D|$ denotes the measure of $D$.
\end{lemma}
\begin{proof}
The first inequality follows from~\eqref{equ:mhj 1} and the
second can be obtained by integrating $\vecm_h^{(j)}(\vecx)$ over $D$.
\end{proof}
\begin{lemma}\label{lem:3.2}
There exist a deterministic constant $c$
depending on $\vecm_0$, $\{\vecg_i\}_{i=1}^q$, $\lambda_1$ and $\lambda_2$
such that  for any $\theta\in[1/2,1]$ and  for $j=1,\ldots,J$ there holds
\begin{align*}
\mE
\left \| \nabla \vecm_h^{(j)} \right \| _{\mL^2(D)}^2
+
k\sum_{i=1}^{j-1}
\mu^{-1}\lambda_2\mE\left \| \vecv_h^{(i)}\right\|_{\mL^2(D)}^2
+
k^2(2\theta-1)\sum_{i=1}^{j-1}
\mE\left \| \nabla\vecv_h^{(i)}\right\|_{\mL^2(D)}^2
\leq 
c.
\end{align*}
\end{lemma}
\begin{proof}
Taking $\vecw_h^{(j)}=\vecv_h^{(j)}$
in equation~\eqref{E:1.5} yields to the following identity:
\begin{align*}
-\lambda_2\left \| \vecv_h^{(j)}\right\|_{\mL^2(D)}^2
=
\mu \inpro{\nabla\vecm_h^{(j)}}{\nabla \vecv_h^{(j)}}_{\mL^2(D)}
+
\mu k \theta \left \| \nabla\vecv_h^{(j)}\right\|_{\mL^2(D)}^2
+\mu F\bigl(t_j,\vecm_h^{(j)}, \vecv_h^{(j)}\bigr),
\end{align*}
or equivalently
\begin{align}\label{eq: boundApp1}
\mu \inpro{\nabla\vecm_h^{(j)}}{\nabla \vecv_h^{(j)}}_{\mL^2(D)}
=
-\lambda_2\left \| \vecv_h^{(j)}\right\|_{\mL^2(D)}^2
-
\mu k \theta \left \| \nabla\vecv_h^{(j)}\right\|_{\mL^2(D)}^2
-
\mu F\bigl(t_j,\vecm_h^{(j)}, \vecv_h^{(j)}\bigr).
\end{align}
From Lemma~\ref{lem:bar} it follows that
\[
\left\|\nabla\vecm_h^{(j+1)}\right\|_{\mL^2(D)}^2
\leq
\left\|\nabla(\vecm_h^{(j)}+k\vecv_h^{(j)})\right\|_{\mL^2(D)}^2,
\]
or equivalently,
\[
\left \| \nabla \vecm_h^{(j+1)} \right \| _{\mL^2(D)}^2
\leq
\left \| \nabla \vecm_h^{(j)} \right \|_{\mL^2(D)}^2
+ k ^2 \left \| \nabla \vecv_h^{(j)} \right \| _{\mL^2(D)}^2
+
2k \inpro{\nabla\vecm_h^{(j)}}{\nabla \vecv_h^{(j)}}_{\mL^2(D)}.
\]
Therefore, together with~\eqref{eq: boundApp1}, we deduce
\begin{align*}
\left \| \nabla \vecm_h^{(j+1)} \right \| _{\mL^2(D)}^2
+
2k\mu^{-1}\lambda_2\left \| \vecv_h^{(j)}\right\|_{\mL^2(D)}^2
+
k^2(2\theta-1)
&\left \| \nabla\vecv_h^{(j)}\right\|_{\mL^2(D)}^2
\leq
\left \| \nabla \vecm_h^{(j)} \right \|_{\mL^2(D)}^2\\
&-
2k 
F\bigl(t_j,\vecm_h^{(j)}, \vecv_h^{(j)}\bigr).
\end{align*}
Thus, it follows from~\eqref{eq: boundExF} that
\begin{align*}
\mE
\left \| \nabla \vecm_h^{(j+1)} \right \| _{\mL^2(D)}^2
&+
2k\mu^{-1}\lambda_2\mE\left \| \vecv_h^{(j)}\right\|_{\mL^2(D)}^2
+
k^2(2\theta-1)
\mE\left \| \nabla\vecv_h^{(j)}\right\|_{\mL^2(D)}^2\\
&\leq
\mE\left \| \nabla \vecm_h^{(j)} \right \|_{\mL^2(D)}^2
+
2k \mE\left[
\bigl|F\bigl(t_j,\vecm_h^{(j)}, \vecv_h^{(j)}\bigr)\big|\right]\\
&\leq 
(1+kc\epsilon T)\mE\left \| \nabla \vecm_h^{(j)} \right \|_{\mL^2(D)}^2
+ 
ck\epsilon (T+T^{1/2})\mE\left \|\vecm_h^{(j)} \right \|_{\mL^2(D)}^2\\
&\quad+ 
ck\epsilon^{-1} (T+T^{1/2}+1)\mE\left \|\vecv_h^{(j)} \right \|_{\mL^2(D)}^2.
\end{align*}
By choosing $\epsilon = \frac{\mu^{-1}\lambda_2}{c(T+T^{1/2}+1)}$ in the right hand side 
of this inequality and using Lemma~\ref{lem:mhj} we deduce
\begin{align*}
\mE
\left \| \nabla \vecm_h^{(j+1)} \right \| _{\mL^2(D)}^2
&+
k\mu^{-1}\lambda_2\mE\left \| \vecv_h^{(j)}\right\|_{\mL^2(D)}^2
+
k^2(2\theta-1)
\mE\left \| \nabla\vecv_h^{(j)}\right\|_{\mL^2(D)}^2\\
&\leq 
ck+
(1+kc)\mE\left \| \nabla \vecm_h^{(j)} \right \|_{\mL^2(D)}^2.
\end{align*}
Replacing $j$ by $i$ in the above inequality and summing for $i$ from $0$ 
to $j-1$ yeilds
\begin{align}\label{eq: boundApp2}
\mE
\left \| \nabla \vecm_h^{(j)} \right \| _{\mL^2(D)}^2
&+
k\sum_{i=1}^{j-1}
\mu^{-1}\lambda_2\mE\left \| \vecv_h^{(i)}\right\|_{\mL^2(D)}^2
+
k^2(2\theta-1)\sum_{i=1}^{j-1}
\mE\left \| \nabla\vecv_h^{(i)}\right\|_{\mL^2(D)}^2\nn\\
&\leq 
ckj+
c\| \nabla \vecm_h^{(0)} \|_{\mL^2(D)}^2
+
kc\sum_{i=1}^{j-1}
\mE\left \| \nabla \vecm_h^{(i)} \right \|_{\mL^2(D)}^2.
\end{align}
Since $\vecm_0\in\mH^2(D)$ it can be shown that there
exists a deterministic constant $c$ depending only on $\vecm_0$
such that
\begin{equation}\label{equ:cm0}
\norm{\nabla\vecm_h^{(0)}}{\mL^2(D)}
\le c.
\end{equation}
Hence, inequality~\eqref{eq: boundApp2} implies
\begin{align}\label{eq: boundApp3}
\mE
\left \| \nabla \vecm_h^{(j)} \right \| _{\mL^2(D)}^2
+
k\sum_{i=1}^{j-1}
\mu^{-1}\lambda_2\mE\left \| \vecv_h^{(i)}\right\|_{\mL^2(D)}^2
&+
k^2(2\theta-1)\sum_{i=1}^{j-1}
\mE\left \| \nabla\vecv_h^{(i)}\right\|_{\mL^2(D)}^2\nn\\
&\leq 
c
+
kc\sum_{i=1}^{j-1}
\mE\left \| \nabla \vecm_h^{(i)} \right \|_{\mL^2(D)}^2.
\end{align}
By using induction and~\eqref{equ:cm0} we can show that 
\[
\mE
\norm{\nabla\vecm_h^i}{\mL^2(D)}^2
\le
c(1+ck)^i.
\]
Summing over $i$ from 0 to $j-1$ and using $1+x\le e^x$ we
obtain
\[
k\sum_{i=0}^{j-1}\mE
\left\|\nabla\vecm_h^i\right \|_{\mL^2(D)}^2
\le
ck \frac{(1+ck)^j-1}{ck}
\le
e^{ckJ} = c.
\]
This together with~\eqref{eq: boundApp3} gives the desired
result.
\end{proof}
\section{The main result}\label{sec:pro}
In this section, we will use the finite element function $\vecm_h^{(j)}$ to construct a sequence of functions that converges in an appropriate sense to a function that is a weak martingale solution of~\eqref{E:1.1} in the sense of Definition~\ref{def:wea sol}.

The discrete solutions $\vecm_h^{(j)}$ and $\vecv_h^{(j)}$
constructed via Algorithm~\ref{Algo:1}
are interpolated in time in the following definition.
\begin{definition}\label{def:mhk}
For all $x\in D$, $\vecu,\vecv\in\mV_h$ and all $t\in[0,T]$, let
$j\in \{ 0,...,J \}$ be
such that  $t \in [t_j, t_{j+1})$. We then define
\begin{align*}
\vecm_{h,k}(t,x)
&:=
\frac{t-t_j}{k}\vecm_h^{(j+1)}(x)
+
\frac{t_{j+1}-t}{k}\vecm_h^{(j)}(x), \\
\vecm_{h,k}^{-}(t,x)
&:=
\vecm_h^{(j)}(x), \\
\vecv_{h,k}(t,x)
&:=
\vecv_h^{(j)}(x),\\
F_{k}(t,\vecu,\vecv)
&:=
F(t_j,\vecu,\vecv)\quad\mP-\text{a.s.}.
\end{align*}
\end{definition}
We note that  $\vecm_{h,k}(t)$ is an $\cF_{t_j}$ adapted process for $t\in[t_j,t_{j+1})$.
The above sequences have the following obvious bounds.
\begin{lemma}\label{lem:3.2a}
There exist a deterministic constant $c$
depending on $\vecm_0$, $\vecg$, $\mu_1$, $\mu_2$ and $T$
such that for all $\theta\in[1/2,1]$,
\begin{align*}
\mE\norm{\vecm_{h,k}^*}{\mL^2(D_T)}^2 +
\mE\left \| \nabla \vecm_{h,k}^* \right \|_{\mL^2(D_T)}^2
+
\mE\left \| \vecv_{h,k}\right\|_{\mL^2(D_T)}^2
+
k (2\theta-1)
\mE\left\| \nabla \vecv_{h,k} \right\|_{\mL^2(D_T)}^2
\leq
c,
\end{align*}
where $\vecm_{h,k}^*=\vecm_{h,k}$ or $\vecm_{h,k}^-$.
In particular,
when $\theta\in[0,\frac{1}{2})$,
\begin{align*}
\mE\norm{\vecm_{h,k}^*}{\mL^2(D_T)}^2
+
&\mE\left \| \nabla \vecm_{h,k}^* \right \| _{\mL^2(D_T)}^2
+
\big(1+(2\theta-1)kh^{-2}\big)
\mE\left \| \vecv_{h,k}\right\|_{\mL^2(D_T)}^2
\leq
c.
\end{align*}
\end{lemma}
\begin{proof}
It is easy to see that
\[
\norm{\vecm_{h,k}^-}{\mL^2(D_T)}^2
=
k \sum_{i=0}^{J-1} \norm{\vecm_{h}^{(i)}}{\mL^2(D)}^2
\quad\text{and}\quad
\norm{\vecv_{h,k}}{\mL^2(D_T)}^2
=
k \sum_{i=0}^{J-1} \norm{\vecv_{h}^{(i)}}{\mL^2(D)}^2.
\]
Both inequalities are direct consequences of
Definition~\ref{def:mhk},
Lemmas~\ref{lem:mhj}, and~\ref{lem:3.2}, noting that
the second inequality requires the use of
the inverse estimate (see e.g.~\cite{Johnson87})
\[
\norm{\nabla\vecv_{h}^{(i)}}{\mL^2(D)}^2
\leq
ch^{-2}
\norm{\vecv_{h}^{(i)}}{\mL^2(D)}^2.
\]
\end{proof}

The next lemma provides a bound of $\vecm_{h,k}$ in the
$\mH^1$-norm and establishes relationships between $\vecm_{h,k}^-$,
$\vecm_{h,k}$ and $\vecv_{h,k}$.
\begin{lemma}\label{lem:3.4}
Assume that $h$ and $k$ approach $0$, with the further condition $k=o(h^2)$ when $\theta\in[0,\frac{1}{2})$. The sequences $\{\vecm_{h,k}\}$, $\{\vecm_{h,k}^{-}\}$, and
$\{\vecv_{h,k}\}$ defined in
Definition~\ref{def:mhk} satisfy the following properties:
\begin{align}
\mE\norm{\vecm_{h,k}}{\mH^1(D_T)}^2
&\le c, \label{equ:mhk h1} \\
\mE\norm{\vecm_{h,k}-\vecm_{h,k}^-}{\mL^2(D_T)}^2
&\le ck^2, \label{equ:mhk mhkm} \\
\mE\norm{\vecv_{h,k}-\pa_t\vecm_{h,k}}{\mL^1(D_T)}
&\le ck, \label{equ:vhk mhk} \\
\mE\norm{|\vecm_{h,k}|-1}{\mL^2(D_T)}^2
&\le c(h^2+k^2). \label{equ:mhk 1}
\end{align}
\end{lemma}
\begin{proof}
The results can be obtained by using Lemma~\ref{lem:3.2a} and the arguments in the proof 
of~\cite[Lemma 6.3]{BNT2016}.
\end{proof}
We now prove 
some properties of $F$ and $F_k$, which will be used in the next two lemmas.
\begin{lemma}\label{lem: Fbound}
For any $\vecu,\vecv\in L^2\bigl(\Omega;L^2(0,T;\mH^1(D))\bigr)$, there exists a constant $c$ depending on 
$T$ and $\{\vecg_i\}_{i=1,\cdots,q}$ such that 
\begin{equation}\label{eq: Fbound10}
\mE[
\int_0^T 
|F^*(t,\vecu(t,\cdot),\vecv(t,\cdot))|\dt
] 
\leq
c
\bigl(\mE[
\|\vecu\|_{\mL^2(D_T)}^2]\bigr)^{1/2}
\bigg(
\bigl(\mE[\|\nabla\vecv\|_{\mL^2(D_T)}^2
]\bigr)^{1/2}
+
\bigl(\mE[\|\vecv\|_{\mL^2(D_T)}^2
]\bigr)^{1/2} \bigg),
\end{equation}
here, $F^* = F$ or $F_k$.
Furthermore,
\begin{align}\label{eq: Fbound12}
\mE[
\int_0^T 
&|F(t,\vecu(t,\cdot),\vecv(t,\cdot))-F_k(t,\vecu(t,\cdot),\vecv(t,\cdot))|\dt
] \nn\\
&\leq
ck^{1/2}
\bigl(\mE[
\|\vecu\|_{\mL^2(D_T)}^2]\bigr)^{1/2}
\bigg(
\bigl(\mE[\|\nabla\vecv\|_{\mL^2(D_T)}^2
]\bigr)^{1/2}
+
\bigl(\mE[\|\vecv\|_{\mL^2(D_T)}^2
]\bigr)^{1/2} \bigg).
\end{align}
\end{lemma}
\begin{proof}
\underline{Proof of~\eqref{eq: Fbound10}:}
The first result of the lemma for $F^* = F$ can be deduced from
Lemma~\ref{lem: Fbound9} by replacing $s\equiv t$, $\vecu\equiv \vecu(t,\cdot)$, 
$\vecv\equiv \vecv(t,\cdot)$  and using H\"older's inequality as follows:
\begin{align}\label{eq: Fbound11}
\mE[
\int_0^T 
&|F(t,\vecu(t,\cdot),\vecv(t,\cdot))|\dt
]
=
\int_0^T \mE[
|F(t,\vecu(t,\cdot),\vecv(t,\cdot))|]\dt\nn\\
&\leq
c\int_0^T
\bigl(\mE[
\|\vecu(t,\cdot)\|_{\mL^2(D)}^2]\bigr)^{1/2}
\bigl(\mE[\|\nabla\vecv(t,\cdot)\|_{\mL^2(D)}^2
]\bigr)^{1/2}\dt\nn\\
&\quad+
c\int_0^T
\bigl(\mE[
\|\vecu(t,\cdot)\|_{\mL^2(D)}^2]\bigr)^{1/2}
\bigl(\mE[\|\vecv(t,\cdot)\|_{\mL^2(D)}^2
]\bigr)^{1/2} \dt\\
&\leq c
\bigl(\mE[
\|\vecu\|_{\mL^2(D_T)}^2]\bigr)^{1/2}
\bigg(
\bigl(\mE[\|\nabla\vecv\|_{\mL^2(D_T)}^2
]\bigr)^{1/2}
+
\bigl(\mE[\|\vecv\|_{\mL^2(D_T)}^2
]\bigr)^{1/2} \bigg).\nn
\end{align}

We first note that 
\begin{align*}
\mE[
\int_0^T 
|F_k(t,\vecu(t,\cdot),\vecv(t,\cdot))|\dt
]
&=
\int_0^T \mE[
|F_k(t,\vecu(t,\cdot),\vecv(t,\cdot))|]\dt\nn\\
&=
\sum_{j=0}^{J-1}
\int_{t_j}^{t_{j+1}} \mE[
|F(t_j,\vecu(t,\cdot),\vecv(t,\cdot))|]\dt,
\end{align*}
then apply Lemma~\ref{lem: Fbound9} for $ s\equiv t_j$, 
$\vecu\equiv \vecu(t,\cdot)$ and $\vecv\equiv \vecv(t,\cdot)$ to deduce
\begin{align*}
\mE[
\int_0^T 
&|F_k(t,\vecu(t,\cdot),\vecv(t,\cdot))|\dt
]
\leq 
c\sum_{j=0}^{J-1}
t_j\int_{t_j}^{t_{j+1}}
\bigl(\mE[
\|\vecu(t,\cdot) \|_{\mL^2(D)}^2]\bigr)^{1/2}
\bigl(\mE
[\|\nabla\vecv(t,\cdot) \|_{\mL^2(D)}^2]\bigr)^{1/2}\dt\nn\\
&\quad+
c\sum_{j=0}^{J-1}
( t_j^{1/2}+ t_j)\int_{t_j}^{t_{j+1}}
\bigl(\mE[
\|\vecu(t,\cdot) \|_{\mL^2(D)}^2]\bigr)^{1/2}
\bigl(\mE[\|\vecv(t,\cdot) \|_{\mL^2(D)}^2
]\bigr)^{1/2}\dt\nn\\
&\leq 
cT\int_0^T
\bigl(\mE[
\|\vecu(t,\cdot) \|_{\mL^2(D)}^2]\bigr)^{1/2}
\bigl(\mE
[\|\nabla\vecv(t,\cdot) \|_{\mL^2(D)}^2]\bigr)^{1/2}\dt\nn\\
&+
c(T+T^{1/2})
\int_0^T
\bigl(\mE[
\|\vecu(t,\cdot) \|_{\mL^2(D)}^2]\bigr)^{1/2}
\bigl(\mE[\|\vecv(t,\cdot) \|_{\mL^2(D)}^2
]\bigr)^{1/2}\dt.
\end{align*}
Hence, ~\eqref{eq: Fbound10} with function $F^*=F_k$ follows by using H\"older's inequality.

\underline{Proof of~\eqref{eq: Fbound12}:}
Noting that 
\begin{align}\label{eq: Fbound13}
\mE[
&\int_0^T 
|F(t,\vecu(t,\cdot),\vecv(t,\cdot))-F_k(t,\vecu(t,\cdot),\vecv(t,\cdot))|\dt\
]\nn\\
&=
\sum_{j=0}^J
\int_{t_j}^{t_{j+1}} 
\mE[
|F(t,\vecu(t,\cdot),\vecv(t,\cdot))-F(t_j,\vecu(t,\cdot),\vecv(t,\cdot))|]\dt\nn\\
&:=
\sum_{j=0}^J
\int_{t_j}^{t_{j+1}} 
\mE[
|\tilde{F}^j(t-t_j,\vecu(t,\cdot),\vecv(t,\cdot))|]\dt.
\end{align}
Here 
\[
\tilde{F}^j(t,\vecx,\vecy) :=  
\sum_{i=1}^q
\int_0^t
\tilde{F}^j_{1,i}(s,\vecx,\vecy)\ds
+
\sum_{i=1}^q\int_0^t
\tilde{F}^j_{2,i}(s,\vecx,\vecy) \,d\tilde{W}_i(s),
\]
in which
$\tilde{F}^j_{1,i}(s,\vecx,\vecy) =  F_{1,i}(s+t_j,\vecx,\vecy) $, 
$\tilde{F}^j_{2,i}(s,\vecx,\vecy) =  F_{2,i}(s+t_j,\vecx,\vecy)$ 
and $\tilde{W}_i(s) = W_i(s+t_j) - W_i(t_j)$. 
By using the same arguments as in the proof of Lemma~\ref{lem: Fbound9} we obtain the same result for 
the upper bound of $\tilde{F}^j$, namely
\begin{align*}
\mE
|\tilde{F}^j( s,\vecu ,\vecv )| 
&\leq 
c s \bigl(\mE[
\|\vecu \|_{\mL^2(D)}^2]\bigr)^{1/2}
\bigl(\mE
[\|\nabla\vecv \|_{\mL^2(D)}^2]\bigr)^{1/2}\nn\\
&\quad+
c( s^{1/2}+ s)\bigl(\mE[
\|\vecu \|_{\mL^2(D)}^2]\bigr)^{1/2}
\bigl(\mE[\|\vecv \|_{\mL^2(D)}^2
]\bigr)^{1/2}.
\end{align*}
Hence, there holds
\begin{align}\label{eq: Fbound14}
\int_{t_j}^{t_{j+1}} 
\mE
|\tilde{F}^j( s,\vecu ,\vecv )| \ds
&\leq 
c \int_{t_j}^{t_{j+1}}s  \bigl(\mE[
\|\vecu \|_{\mL^2(D)}^2]\bigr)^{1/2}
\bigl(\mE
[\|\nabla\vecv \|_{\mL^2(D)}^2]\bigr)^{1/2}\ds\nn\\
&\quad+
c\int_{t_j}^{t_{j+1}}( s^{1/2}+ s)\bigl(\mE[
\|\vecu \|_{\mL^2(D)}^2]\bigr)^{1/2}
\bigl(\mE[\|\vecv \|_{\mL^2(D)}^2
]\bigr)^{1/2}\ds,\nn\\
&\leq 
ck \int_{t_j}^{t_{j+1}}  \bigl(\mE[
\|\vecu \|_{\mL^2(D)}^2]\bigr)^{1/2}
\bigl(\mE
[\|\nabla\vecv \|_{\mL^2(D)}^2]\bigr)^{1/2}\ds\nn\\
&\quad+
c( k^{1/2}+ k)\int_{t_j}^{t_{j+1}}\bigl(\mE[
\|\vecu \|_{\mL^2(D)}^2]\bigr)^{1/2}
\bigl(\mE[\|\vecv \|_{\mL^2(D)}^2
]\bigr)^{1/2}\ds.
\end{align}
Therefore, it follows from~\eqref{eq: Fbound13} and~\eqref{eq: Fbound14} that 
\begin{align*}
\mE[
&\int_0^T 
|F(t,\vecu(t,\cdot),\vecv(t,\cdot))-F_k(t,\vecu(t,\cdot),\vecv(t,\cdot))|\dt\
]\nn\\
&\leq 
ck \int_0^T  \bigl(\mE[
\|\vecu \|_{\mL^2(D)}^2]\bigr)^{1/2}
\bigl(\mE
[\|\nabla\vecv \|_{\mL^2(D)}^2]\bigr)^{1/2}\ds\nn\\
&\quad+
c( k^{1/2}+ k)\int_0^T\bigl(\mE[
\|\vecu \|_{\mL^2(D)}^2]\bigr)^{1/2}
\bigl(\mE[\|\vecv \|_{\mL^2(D)}^2
]\bigr)^{1/2}\ds.
\end{align*}
The result follows immediately by using H\"older's inequality, which completes the proof of the lemma.
\end{proof}

The following two Lemmas~\ref{lem:3.7} and ~\ref{lem:3.7a} show that ~$\vecm_{h,k}^-$ and~$\vecm_{h,k}$, 
respectively, satisfy a discrete form of~\eqref{InE:13}.
\begin{lemma}\label{lem:3.7}
Assume that $h$ and $k$ approach $0$ with the following
conditions
\begin{equation}\label{equ:theta}
\begin{cases}
k = o(h^2) & \quad\text{when } 0 \le \theta < 1/2, \\
k = o(h) & \quad\text{when } \theta = 1/2, \\
\text{no condition} & \quad\text{when } 1/2<\theta\le1.
\end{cases}
\end{equation}
Then for any $\vecpsi \in C_0^\infty\big((0,T);\C^{\infty}(D)\big)$, there holds $\mP$-a.s.
\begin{align*}
&-\lambda_1\inpro{\vecm_{h,k}^-\times\vecv_{h,k}}
{\vecm_{h,k}^-\times\vecpsi}_{\mL^2(D_T)}
+
\lambda_2\inpro{\vecv_{h,k}}
{\vecm_{h,k}^-\times\vecpsi}_{\mL^2(D_T)}
\nonumber\\
&+
\mu\inpro{\nabla(\vecm_{h,k}^-+k\theta\vecv_{h,k})}
{\nabla(\vecm_{h,k}^-\times\vecpsi)}_{\mL^2(D_T)}
+
\mu\int_0^TF_k(t,\vecm_{h,k}^-,\vecm_{h,k}^-\times\vecpsi)\dt
=\sum_{j=1}^3 I_j,
\end{align*}
where
\begin{align*}
I_1
&:=
\inpro{-\lambda_1\vecm_{h,k}^-\times\vecv_{h,k}
+
\lambda_2\vecv_{h,k}}
{\vecm_{h,k}^-\times\vecpsi
-
I_{\mV_h}(\vecm_{h,k}^-\times\vecpsi)}_{\mL^2(D_T)},\\
I_2
&:=
\mu\inpro{\nabla(\vecm_{h,k}^-+k\theta\vecv_{h,k})}
{\nabla(\vecm_{h,k}^-\times\vecpsi-
I_{\mV_h}(\vecm_{h,k}^-\times\vecpsi))}_{\mL^2(D_T)},\\
I_3
&:=\mu
\int_0^T
F_k(t,\vecm_{h,k}^-,\vecm_{h,k}^-\times\vecpsi)
-
F_k(t,\vecm_{h,k}^-,I_{\mV_h}(\vecm_{h,k}^-\times\vecpsi))\dt.
\end{align*}
Furthermore, $\mE |I_i|=O(h)$ for $i=1,2,3$.
\end{lemma}
\begin{proof}
For $t\in[t_j, t_{j+1})$, we use
equation~\eqref{E:1.5} with
$\vecw_h^{(j)}=I_{\mV_h}\big(\vecm_{h,k}^-(t,\cdot)\times\vecpsi(t,\cdot)\big)$
to see
\begin{align*}
-\lambda_1
&\inpro{\vecm_{h,k}^-(t,\cdot)\times\vecv_{h,k}(t,\cdot)}
{I_{\mV_h}\big(\vecm_{h,k}^-(t,\cdot)\times\vecpsi(t,\cdot)\big)}_{\mL^2(D)} \\
&+
\lambda_2\inpro{\vecv_{h,k}(t,\cdot)}
{I_{\mV_h}\big(\vecm_{h,k}^-(t,\cdot)\times\vecpsi(t,\cdot)\big)}_{\mL^2(D)}\nonumber\\
&+
\mu\inpro{\nabla(\vecm_{h,k}^-(t,\cdot)+k\theta\vecv_{h,k}(t,\cdot))}
{\nabla
I_{\mV_h}\big(\vecm_{h,k}^-(t,\cdot)\times\vecpsi(t,\cdot)\big)}_{\mL^2(D)}\\
&+\mu
F_k(t,\vecm_{h,k}^-(t,\cdot),I_{\mV_h}\big(\vecm_{h,k}^-(t,\cdot)\times\vecpsi(t,\cdot)\big))
= 0.
\end{align*}
Integrating both sides of the above equation over
$(t_j,t_{j+1})$ and summing over $j=0,\ldots,J-1$ we deduce
\begin{align*}
-&\lambda_1
\inpro{\vecm_{h,k}^-\times\vecv_{h,k}}
{I_{\mV_h}\big(\vecm_{h,k}^-\times\vecpsi\big)}_{\mL^2(D_T)}
+
\lambda_2\inpro{\vecv_{h,k}}
{I_{\mV_h}\big(\vecm_{h,k}^-\times\vecpsi\big)}_{\mL^2(D_T)}\nonumber\\
&+
\mu\inpro{\nabla(\vecm_{h,k}^-+k\theta\vecv_{h,k})}
{\nabla I_{\mV_h}\big(\vecm_{h,k}^-\times\vecpsi\big)}_{\mL^2(D_T)}
+
\mu\int_0^T
F_k(t,\vecm_{h,k}^-,I_{\mV_h}(\vecm_{h,k}^-\times\vecpsi))\dt
= 0.
\end{align*}
This implies
\begin{align*}
-\lambda_1
&\inpro{\vecm_{h,k}^-\times\vecv_{h,k}}
{\vecm_{h,k}^-\times\vecpsi}_{\mL^2(D_T)}
+
\lambda_2\inpro{\vecv_{h,k}}
{\vecm_{h,k}^-\times\vecpsi}_{\mL^2(D_T)}\nonumber\\
&+
\mu\inpro{\nabla(\vecm_{h,k}^-+k\theta\vecv_{h,k})}
{\nabla(\vecm_{h,k}^-\times\vecpsi)}_{\mL^2(D_T)}
+
\mu\int_0^T
F_k(t,\vecm_{h,k}^-,\vecm_{h,k}^-\times\vecpsi)\dt\nn\\
&=I_1+I_2+I_3.
\end{align*}

Hence it suffices to prove that $\mE |I_i|=O(h)$ for $i=1,2,3$.
First, by using Lemma~\ref{lem:mhj} we obtain
\begin{equation}\label{eq: mhkinfty}
\norm{\vecm_{h,k}^-}{\mL^\infty(D_T)}
\le
\sup_{0\le j \le J} \norm{\vecm_h^{(j)}}{\mL^\infty(D)}
\leq
1,
\end{equation}
and 
\begin{equation}\label{eq: mhkinfty2}
\norm{\vecm_{h,k}}{\mL^\infty(D_T)}
\leq 2
\sup_{0\leq j \leq J} \norm{\vecm_h^{(j)}}{\mL^\infty(D)}
\leq
2.
\end{equation}
Lemma~\ref{lem:3.2a} and~\eqref{eq: mhkinfty} together with H\"older's
inequality and Lemma~\ref{lem:Ih vh} yield
\begin{align*}
\mE|I_1|
&\le
c\mE\left[
\left(\norm{\vecm_{h,k}^-}{\mL^\infty(D_T)}+1\right)
\norm{\vecv_{h,k}}{\mL^2(D_T)}
\norm{\vecm_{h,k}^-\times\vecpsi
-
I_{\mV_h}(\vecm_{h,k}^-\times\vecpsi)}{\mL^2(D_T)}\right] \\
&\le
c\mE\left[\norm{\vecv_{h,k}}{\mL^2(D_T)}
\norm{\vecm_{h,k}^-\times\vecpsi
-
I_{\mV_h}(\vecm_{h,k}^-\times\vecpsi)}{\mL^2(D_T)}\right] \\
&\leq
c\bigl(\mE[\norm{\vecv_{h,k}}{\mL^2(D_T)}^2]\bigr)^{1/2}
\bigl(\mE[\norm{\vecm_{h,k}^-\times\vecpsi
-
I_{\mV_h}(\vecm_{h,k}^-\times\vecpsi)}{\mL^2(D_T)}^2]\bigr)^{1/2}
\le
ch.
\end{align*}
The bound for $\mE|I_2|$ can be obtained similarly,
using Lemma~\ref{lem:3.2a} and noting that when $\theta\in[0,\frac{1}{2}]$, 
a suitable bound on $k\left\| \nabla \vecv_{h,k} \right\|_{\mL^2(D_T)}$ can be deduced from the inverse estimate as follows:
\[
k\left\| \nabla \vecv_{h,k} \right\|_{\mL^2(D_T)}
\leq
ckh^{-1}
\left \| \vecv_{h,k}\right\|_{\mL^2(D_T)}
\leq ckh^{-1}.
\]
The bound for  $\mE|I_3|$ can be obtained by noting the linearity of $F$ in Remark~\ref{rem: Fsym} 
and using Lemmas~\ref{lem: Fbound} and~\ref{lem:Ih vh}. Indeed, 
\begin{align*}
\mE|I_3|
&=
\mu\mE\bigl|
\int_0^T
F_k(t,\vecm_{h,k}^-,\vecm_{h,k}^-\times\vecpsi
-I_{\mV_h}(\vecm_{h,k}^-\times\vecpsi))\dt\bigr|\\
&\leq 
\mu\mE
\int_0^T\bigl|
F_k(t,\vecm_{h,k}^-,\vecm_{h,k}^-\times\vecpsi
-I_{\mV_h}(\vecm_{h,k}^-\times\vecpsi))\bigr|\dt\\
&\leq 
c
\bigl(\mE[
\|\vecm_{h,k}^-\|_{\mL^2(D_T)}^2]\bigr)^{1/2}
\bigg(
\bigl(\mE[\|\nabla\bigl(\vecm_{h,k}^-\times\vecpsi
-I_{\mV_h}(\vecm_{h,k}^-\times\vecpsi)\bigr)\|_{\mL^2(D_T)}^2
]\bigr)^{1/2}\\
&\quad\quad+
\bigl(\mE[\|\vecm_{h,k}^-\times\vecpsi
-I_{\mV_h}(\vecm_{h,k}^-\times\vecpsi)\|_{\mL^2(D_T)}^2
]\bigr)^{1/2} \bigg)\\
&\leq ch.
\end{align*}
This completes the proof of the lemma.
\end{proof}
\begin{lemma}\label{lem:3.7a}
Assume that $h$ and $k$ approach 0 satisfying~\eqref{equ:theta}.
Then for any $\vecpsi \in C_0^\infty\big((0,T);\C^{\infty}(D)\big)$, 
there holds $\mP$-a.s.
\begin{align}\label{InE:10}
&-\lambda_1
\inpro{\vecm_{h,k}\times\pa_t\vecm_{h,k}}
{\vecm_{h,k}\times\vecpsi}_{\mL^2(D_T)}
+
\lambda_2\inpro{\pa_t\vecm_{h,k}}
{\vecm_{h,k}\times\vecpsi}_{\mL^2(D_T)}
\nonumber\\
&+
\mu\inpro{\nabla(\vecm_{h,k})}
{\nabla(\vecm_{h,k}\times\vecpsi)}_{\mL^2(D_T)}
+
\mu\int_0^TF(t,\vecm_{h,k},\vecm_{h,k}\times\vecpsi)\dt
= \sum_{j=1}^7I_j,
\end{align}
where
\begin{align*}
I_4
&=
-\lambda_1\inpro{\vecm_{h,k}^-\times\vecv_{h,k}}
{\vecm_{h,k}^-\times\vecpsi}_{\mL^2(D_T)}
+
\lambda_1\inpro{\vecm_{h,k}\times\pa_t\vecm_{h,k}}
{\vecm_{h,k}\times\vecpsi}_{\mL^2(D_T)},\\
I_5
&=
\lambda_2\inpro{\vecv_{h,k}}
{\vecm_{h,k}^-\times\vecpsi}_{\mL^2(D_T)}
-
\lambda_2\inpro{\pa_t\vecm_{h,k}}
{\vecm_{h,k}\times\vecpsi}_{\mL^2(D_T)}
,\\
I_6
&=
\mu\inpro{\nabla(\vecm_{h,k}^-+k\theta\vecv_{h,k})}
{\nabla(\vecm_{h,k}^-\times\vecpsi)}_{\mL^2(D_T)}
-
\mu\inpro{\nabla(\vecm_{h,k})}
{\nabla(\vecm_{h,k}\times\vecpsi)}_{\mL^2(D_T)},
\\
I_7
&=\mu
\int_0^T
\bigl(
F(t,\vecm_{h,k},\vecm_{h,k}\times\vecpsi)
-
F_k(t,\vecm_{h,k}^-,\vecm_{h,k}^-\times\vecpsi)\bigr)\dt.
\end{align*}
Furthermore, $\mE|I_i|=O(h)$ for $i=1,\cdots,6$ and $\mE|I_7|=O(h+k^{1/2})$.
\end{lemma}
\begin{proof}
From Lemma~\ref{lem:3.7} it follows that
\begin{align*}
-\lambda_1&\inpro{\vecm_{h,k}\times\pa_t\vecm_{h,k}}
{\vecm_{h,k}\times\vecpsi}_{\mL^2(D_T)}
+
\lambda_2\inpro{\pa_t\vecm_{h,k}}
{\vecm_{h,k}\times\vecpsi}_{\mL^2(D_T)}
\nonumber\\
&+
\mu\inpro{\nabla(\vecm_{h,k})}
{\nabla(\vecm_{h,k}\times\vecpsi)}_{\mL^2(D_T)}
+
\mu\int_0^TF_k(t,\vecm_{h,k},\vecm_{h,k}\times\vecpsi)\dt \nn\\
&=
I_1+\cdots+I_7.
\end{align*}

Hence it suffices to prove that $\mE|I_i|=O(h)$ for $i=4,\cdots,6$. First, by using the triangle inequality and 
H\"older's inequality, we obtain
\begin{align*}
\lambda_1^{-1}|I_4|
&\leq
\left|
\inpro{(\vecm_{h,k}^--\vecm_{h,k})\times\vecv_{h,k}}
{\vecm_{h,k}^-\times\vecpsi}_{\mL^2(D_T)}
\right|\\
&\quad+
\left|
\inpro{\vecm_{h,k}\times\vecv_{h,k}}
{(\vecm_{h,k}^--\vecm_{h,k})\times\vecpsi}_{\mL^2(D_T)}
\right|\\
&\quad+
\left|
\inpro{\vecm_{h,k}\times(\vecv_{h,k}-\pa_t\vecm_{h,k})}
{\vecm_{h,k}\times\vecpsi}_{\mL^2(D_T)}
\right|,\\
&\leq
2\norm{\vecm_{h,k}^--\vecm_{h,k}}{\mL^2(D_T)}
\norm{\vecv_{h,k}}{\mL^2(D_T)}
\bigl(\norm{\vecm_{h,k}^-}{\mL^{\infty}(D_T)}+\norm{\vecm_{h,k}}{\mL^{\infty}(D_T)}\bigr)
\norm{\vecpsi}{\mL^{\infty}(D_T)}\\
&\quad+
\norm{\vecv_{h,k}-\pa_t\vecm_{h,k}}{\mL^1(D_T)}
\norm{\vecm_{h,k}}{\mL^{\infty}(D_T)}^2
\norm{\vecpsi}{\mL^{\infty}(D_T)}.
\end{align*}
Therefore, the required bound on $\mE|I_4|$ can be obtained by 
using~\eqref{eq: mhkinfty},~\eqref{eq: mhkinfty2} and Lemmas~\ref{lem:3.2a},~\ref{lem:3.4}. 
The bounds on $\mE|I_5|$ and $\mE|I_6|$  can be obtaineded similarly. 

In order to prove the bound for $\mE|I_7|$, we first use the triangle 
inequality then Remark~\ref{rem: Fsym} and Lemma~\ref{lem: Fbound} to obtain
\begin{align*}
\mE|I_7|
&\leq 
\mE
\int_0^T
\left|
F(t,\vecm_{h,k},\vecm_{h,k}\times\vecpsi) - F_k(t,\vecm_{h,k},\vecm_{h,k}\times\vecpsi) \right|\dt\nn\\
&\quad+
\mE
\int_0^T
\left|
F_k(t,\vecm_{h,k}-\vecm_{h,k}^-,\vecm_{h,k}\times\vecpsi)\right|\dt
+
\mE
\int_0^T
\left|
F_k(t,\vecm_{h,k}^-,(\vecm_{h,k}-\vecm_{h,k}^-)\times\vecpsi)\right|\dt\\
&\leq
ck^{1/2}\bigl(\mE[
\|\vecm_{h,k}\|_{\mL^2(D_T)}^2]\bigr)^{1/2}
\bigg(
\bigl(\mE[\|\nabla\vecm_{h,k}\|_{\mL^2(D_T)}^2
]\bigr)^{1/2}
+
\bigl(\mE[\|\vecm_{h,k}\|_{\mL^2(D_T)}^2
]\bigr)^{1/2} \bigg)\nn\\
&\quad+
c\bigl(\mE[
\|\vecm_{h,k}-\vecm_{h,k}^-\|_{\mL^2(D_T)}^2]\bigr)^{1/2}
\bigg(
\bigl(\mE[\|\nabla\vecm_{h,k}\|_{\mL^2(D_T)}^2
]\bigr)^{1/2}
+
\bigl(\mE[\|\vecm_{h,k}\|_{\mL^2(D_T)}^2
]\bigr)^{1/2} \bigg)\\
&\leq c(h+k^{1/2}),
\end{align*}
in which~\eqref{equ:mhk h1} and~\eqref{equ:mhk mhkm} are used to obtain the last 
inequality. This completes the proof of the lemma.
\end{proof}
In order to prove the convergence of random variables $\vecm_{h,k}$, 
we first state a result of tightness for the family $\cL(\vecm_{h,k})$. 
We  then use the Skorohod theorem to define another probability space and an almost surely convergent 
sequence defined in this space whose limit is a weak martingale solution of equation~\eqref{E:1.3ab}. 
The proof of the following results are omitted since they are 
relatively simple modification of the proof of the corresponding results from~\cite{BNT2016}.
\begin{lemma}\label{lem:tig}
Assume that $h$ and $k$ approach $0$, and further that ~\eqref{equ:theta} holds.
Then the set of laws~$\{\cL(\vecm_{h,k})\}$ on the Banach space
$C\big([0,T];\mH^{-1}(D)\big)\cap L^2(0,T;\mL^2(D))$ is tight.
\end{lemma}
\begin{proposition}\label{pro:con}
Assume that $h$ and $k$ approach $0$, and further that ~\eqref{equ:theta} holds.
Then there exist:
\begin{enumerate}
\renewcommand{\labelenumi}{(\alph{enumi})}
\item
a probability space $(\Omega',\cF',\mP')$;
\item
a sequence $\{\vecm'_{h,k}\}$ of random variables
defined on $(\Omega',\cF',\mP')$ and taking values in $C\big([0,T];\mH^{-1}(D)\big)\cap L^2(0,T;\mL^2(D))$; and
\item
a random variable $\vecm'$ defined on
$(\Omega',\cF',\mP')$ and taking values in $C\big([0,T];\mH^{-1}(D)\big)\cap L^2(0,T;\mL^2(D))$,
\end{enumerate}
satisfying
\begin{enumerate}
\item\label{item:a}
$\cL(\vecm_{h,k}) = \cL(\vecm_{h,k}')$,
\item\label{item:b}
$\vecm_{h,k}'\goto\vecm'$ in $C\big([0,T];\mH^{-1}(D)\big)\cap L^2(0,T;\mL^2(D))$ strongly,
$\mP'$-a.s..
\end{enumerate}

Moreover, the sequence $\{\vecm'_{h,k}\}$ satisfies
\begin{align}
\mE[ \|\vecm_{h,k}'\|^2_{\mH(D_T)}]&\leq c\label{eq: m'1},\\
\mE[ \||\vecm_{h,k}'|-1\|^2_{\mL^2(D_T)}]&\leq c(h^2+k^2)\label{eq: m'2},\\
\|\vecm_{h,k}'\|^2_{\mL^{\infty}(D_T)}&\leq c\quad\mP'\text{-a.s.,}\label{eq: m'3}
\end{align}
here, $c$ is a positive constant only depending on $\{\vecg_i\}_{i=1,\cdots,q}$.
\end{proposition}
We are now ready to state and prove our main theorem.
\begin{theorem}\label{the:mai}
Assume that $T>0$, $\vecM_0\in\mH^1(D)$ satisfies~\eqref{equ:m0} and $\vecg_i\in\mW^{2,\infty}(D)$ for 
$i=1,\cdots,q$ satisfy the homogeneous Neumann boundary condition.  
Then 
$\vecm'$, the sequence $\{\vecm'_{h,k}\}$ and the probability space $(\Omega',\cF',\mP')$  given by Proposition~\ref{pro:con} satisfy:
\begin{enumerate}
\item
the sequence of $\{\vecm'_{h,k}\}$ converges to $\vecm'$  weakly in $L^2(\Omega';\mH^1(D_T))$; and
\item
$\big(\Omega',\cF',(\cF'_t)_{t\in[0,T]},\mP',\vecM'\big)$ is a weak martingale solution of~\eqref{E:1.1},
where 
\[
\vecM'(t):=Z_t\vecm'(t)\quad \forall t\in[0,T],\text{ a.e. } \vecx\in D.
\] 
\end{enumerate}
\end{theorem}
\begin{proof}
From~\eqref{eq: m'3} and property (2) of Proposition~\ref{pro:con}, there exists a set $V\subset \Omega'$ 
such that $\mP'(V) = 1$ and for all $\omega'\in V$ there hold
\[
\|\vecm_{h,k}'(\omega')\|^2_{\mL^2(D_T)}\leq c
\quad\text{and}\quad
\vecm_{h,k}'(\omega')\goto\vecm'(\omega') \text{ in } \mL^2(D_T) \text{ strongly}.
\]
Hence, by using Lebesgue's dominated convergence theorem, we deduce
\begin{equation}\label{eq: m'4}
\vecm_{h,k}'\goto\vecm' \text{ in } L^2(\Omega';\mL^2(D_T)) \text{ strongly},
\end{equation}
which implies from~\eqref{eq: m'1} that 
\begin{equation}\label{eq: m'5}
\vecm_{h,k}'\goto\vecm' \text{ in } L^2(\Omega';\mH^1(D_T)) \text{ weakly}.
\end{equation}

In order to prove Part (2), by noting Lemma~\ref{lem:equi} and Remark~\ref{rem:LLL} we only need to prove that 
$\vecm'$ satisfies~\eqref{equ:m 1} and~\eqref{E:1.3ab}, namely
\begin{equation}\label{eq: m'6}
|\vecm'(t,\vecx)| = 1,\quad t\in (0,T),\quad \vecx\in D,\quad \mP'\text{-a.s.}
\end{equation}
and 
\begin{equation}\label{eq: m'7}
\cI(\vecm',\vecvarphi) = 0\quad \mP'\text{-a.s.}\quad \forall \vecvarphi\in L^2(0,T;\mH^1(D)),
\end{equation}
where 
\begin{align*}
\cI(\vecm',\vecvarphi)
:=
&\lambda_1\inpro{\vecm'\times\partial_t\vecm'}{\vecm'\times\vecvarphi}_{\mL^2(D_T)}
-
\lambda_2\inpro{\partial_t\vecm'}{\vecm'\times\vecvarphi}_{\mL^2(D_T)} \nn\\
&-
\mu \inpro{\nabla\vecm'}{\nabla(\vecm'\times\vecvarphi)}_{\mL^2(D_T)}
-
\mu \int_0^T F(t,\vecm'(t,\cdot),\vecm'(t,\cdot)\times\vecvarphi(t,\cdot))\dt.
\end{align*}
By using~\eqref{eq: m'2} and~\eqref{eq: m'4}, we obtain~\eqref{eq: m'6} immediately.

In order to prove~\eqref{eq: m'7}, we first find the equation satisfied by $\vecm_{h,k}'$ and then pass to 
the limit when $h$ and $k$ approach $0$.

By using Lemmas~\ref{lem:3.7a} and property (1) of Proposition~\ref{pro:con}, it follows that 
for any $\vecpsi\in C^{\infty}_0(0,T;\mathbb C^{\infty}(D))$ that there holds
\begin{equation}\label{eq: m'8}
\mE|\cI(\vecm_{h,k}',\vecpsi)| = O(h+k^{1/2}). 
\end{equation}
To pass to the limit in~\eqref{eq: m'8}, we first using~\eqref{eq: m'4}--\eqref{eq: m'6} and 
the same arguments as in~\cite[Theorem 6.8]{BNT2016} to obtain
that as $h$ and $k$ tend to $0$,
\begin{align}
\inpro{\vecm_{h,k}'\times\partial_t\vecm_{h,k}'}{\vecm_{h,k}'\times\vecvarphi}_{L^2(\Omega';\mL^2(D_T))}
&\goto 
\inpro{\vecm'\times\partial_t\vecm'}{\vecm'\times\vecvarphi}_{L^2(\Omega';\mL^2(D_T))},\label{eq: m'9}\\
\inpro{\partial_t\vecm'_{h,k}}{\vecm'_{h,k}\times\vecvarphi}_{L^2(\Omega';\mL^2(D_T))}
&\goto
\inpro{\partial_t\vecm'}{\vecm'\times\vecvarphi}_{L^2(\Omega';\mL^2(D_T))},\label{eq: m'10}\\
\inpro{\nabla\vecm'_{h,k}}{\nabla(\vecm'_{h,k}\times\vecvarphi)}_{L^2(\Omega';\mL^2(D_T))}
&\goto
\inpro{\nabla\vecm'}{\nabla(\vecm'\times\vecvarphi)}_{L^2(\Omega';\mL^2(D_T))}.\label{eq: m'11}
\end{align}
Then, by using Remark~\ref{rem: Fsym} and~\eqref{eq: Fbound10} with $F^*=F$, we estimate
\begin{align*}
\mE\int_0^T\bigl|
&F(t,\vecm_{h,k}'(t,\cdot),\vecm_{h,k}'(t,\cdot)\times\vecvarphi(t,\cdot))
-
F(t,\vecm'(t,\cdot),\vecm'(t,\cdot)\times\vecvarphi(t,\cdot))
\bigr|\dt\nn\\
&\leq 
\mE\int_0^T\bigl|
F(t,\vecm_{h,k}'(t,\cdot)-\vecm'(t,\cdot),\vecm_{h,k}'(t,\cdot)\times\vecvarphi(t,\cdot))\bigr|\dt\nn\\
&\quad+
\mE\int_0^T\bigl|
F(t,\bigl(\vecm_{h,k}'(t,\cdot)-\vecm'(t,\cdot)\bigr)\times\vecvarphi(t,\cdot),\vecm'(t,\cdot))\bigr|\dt\nn\\
&\leq 
c 
\|\vecm_{h,k}'-\vecm'\|_{L^2(\Omega';\mL^2(D_T))}
\bigl(
\|\nabla\vecm_{h,k}'\|_{L^2(\Omega';\mL^2(D_T))}
+
\|\vecm_{h,k}'\|_{L^2(\Omega';\mL^2(D_T))}\nn\\
&\quad\quad+
\|\nabla\vecm'\|_{L^2(\Omega';\mL^2(D_T))}
+
\|\vecm'\|_{L^2(\Omega';\mL^2(D_T))}
\bigr).
\end{align*}
Since $\vecm'\in L^2(\Omega';\mH^1(D_T))$, it follows from~\eqref{eq: m'1} and~\eqref{eq: m'4} that 
\begin{equation}\label{eq: m'12}
\mE\bigl[
\int_0^T F(t,\vecm_{h,k}'(t,\cdot),\vecm_{h,k}'(t,\cdot)\times\vecvarphi(t,\cdot))\dt
\bigr]
\goto
\mE\bigl[
\int_0^T F(t,\vecm'(t,\cdot),\vecm'(t,\cdot)\times\vecvarphi(t,\cdot))\dt
\bigr], 
\end{equation}
as $h$ and $k$ tend to $0$. From~\eqref{eq: m'9}--\eqref{eq: m'12} we deduce that
\[
\mE |\cI(\vecm_{h,k}',\vecpsi) - \cI(\vecm',\vecpsi)|\goto 0,
\]
and hence, together with~\eqref{eq: m'8} $\mE |\cI(\vecm',\vecpsi)| =  0$. This implies~\eqref{eq: m'7} which
completes the proof of our main theorem.
\end{proof}

\section{Appendix}\label{sec:app}
For the reader's convenience we will recall the following results, 
which are proved in~\cite{BNT2016}.
\begin{lemma}\label{lem:4.0}
For any real constants $\lambda_1$ and $\lambda_2$ with
$\lambda_1\not=0$, if $\vecpsi, \veczeta\in\R^3$
satisfy $|\veczeta|=1$, then there exists
$\vecvarphi\in\R^3$ satisfying
\begin{equation}\label{equ:app}
\lambda_1\vecvarphi
+
\lambda_2\vecvarphi\times\veczeta
=\vecpsi.
\end{equation}
As a consequence, if
$\veczeta\in\mH^1(D_T)$ with $|\veczeta(t,x)|=1$ a.e. in
$D_T$ and $\vecpsi\in L^2(0,T;\mW^{1,\infty}(D))$, then
$\vecvarphi\in L^2(0,T;\mH^1(D))$.
\end{lemma}

\begin{lemma}\label{lem:Ih vh}
For any $\vecv\in\C(D)$, $\vecv_h\in\mV_h$ and
$\vecpsi\in\C_0^\infty(D_T)$,
\begin{align*}
\norm{I_{\mV_h}\vecv}{\mL^{\infty}(D)}
&\le
\norm{\vecv}{\mL^{\infty}(D)}, \\
\norm{\vecm_{h,k}^-\times\vecpsi
-
I_{\mV_h}(\vecm_{h,k}^-\times\vecpsi)}{\mL([0,T],\mH^1(D))}^2
&\le
ch^2
\norm{\vecm_{h,k}^-}{\mL([0,T],\mH^1(D))}^2
\norm{\vecpsi}{\mW^{2,\infty}(D_T)}^2,
\end{align*}
where $\vecm_{h,k}^-$ is  defined in Defintion~\ref{def:mhk}
\end{lemma}

The next lemma defines a discrete $\mL^p$-norm in
$\mV_h$, equivalent to the usual $\mL^p$-norm.
\begin{lemma}\label{lem:nor equ}
There exist $h$-independent positive constants $C_1$ and
$C_2$ such that for all $p\in[1,\infty]$ and
$\vecu\in\mV_h$,
\begin{equation*}
C_1\|\vecu\|^p_{\mL^p(\Omega)}
\leq
h^d
\sum_{n=1}^N |\vecu(\vecx_n)|^p
\leq
C_2\|\vecu\|^p_{\mL^p(\Omega)},
\end{equation*}
where $\Omega\subset\R^d$, d=1,2,3.
\end{lemma}
\section*{Acknowledgements}
The authors acknowledge financial support through the ARC Discovery
projects DP140101193 and DP120101886. They are grateful to Vivien Challis for a number of helpful conversations.

\bibliographystyle{myabbrv}
\bibliography{mybib}

\end{document}

%% file: ams_notation.tex
\newcommand{\C}{{\mathbb C}}
\newcommand{\R}{{\mathbb R}}

\newcommand{\mP}{\mathbb P}

\newcommand{\mE}{\mathbb E}
\newcommand{\mH}{\mathbb H}
\newcommand{\mL}{\mathbb L}
\newcommand{\mT}{\mathbb T}
\newcommand{\mV}{\mathbb V}
\newcommand{\mW}{\mathbb W}

\newcommand{\inpro}[2]{\left\langle{#1},{#2}\right\rangle}

\newcommand{\norm}[2]{\|{#1}\|_{#2}}

\newcommand{\snorm}[2]{\left|{#1}\right|_{#2}}

\newcommand{\vecH}{\boldsymbol{H}}

\newcommand{\vecM}{\boldsymbol{M}}

\newcommand{\veca}{\boldsymbol{a}}
\newcommand{\vecb}{\boldsymbol{b}}
\newcommand{\vecc}{\boldsymbol{c}}

\newcommand{\vecg}{\boldsymbol{g}}

\newcommand{\vecm}{\boldsymbol{m}}
\newcommand{\vecn}{\boldsymbol{n}}

\newcommand{\vecu}{\boldsymbol{u}}
\newcommand{\vecv}{\boldsymbol{v}}
\newcommand{\vecw}{\boldsymbol{w}}
\newcommand{\vecx}{\boldsymbol{x}}
\newcommand{\vecy}{\boldsymbol{y}}

\newcommand{\vecvarphi}{\boldsymbol{\varphi}}

\newcommand{\vecpsi}{\boldsymbol{\psi}}
\newcommand{\vecxi}{\boldsymbol{\xi}}
\newcommand{\veczeta}{\boldsymbol{\zeta}}

\DeclareMathOperator*{\esssup}{ess\,sup \/}

\newcommand{\cF}{{\mathcal F}}

\newcommand{\cI}{{\mathcal I}}

\newcommand{\cL}{{\mathcal L}}

\newcommand{\cN}{{\mathcal N}}

\newcommand{\goto}{\rightarrow}

\newcommand{\p}{\partial}

\newcommand{\pa}{\partial}

\newcommand{\ds}{\, ds}
\newcommand{\dt}{\, dt}

\newcommand{\dvx}{\, d\vecx}

\newcommand{\nn}{\nonumber}

%% file: env.tex
\numberwithin{equation}{section}
\newtheorem{theorem}{Theorem}[section]
\newtheorem{lemma}[theorem]{Lemma}

\newtheorem{proposition}[theorem]{Proposition}
\newtheorem{remark}[theorem]{Remark}
\newtheorem{definition}[theorem]{Definition}



%% file: SllG_Multi_Ver10.bbl
\def\cprime{$'$}
\begin{thebibliography}{10}

\bibitem{Alo08}
F.~Alouges.
\newblock A new finite element scheme for {L}andau-{L}ifchitz equations.
\newblock {\em Discrete Contin. Dyn. Syst. Ser. S},  {\bf 1} (2008), 187--196.

\bibitem{Alouges2014}
F.~Alouges, A.~de~Bouard, and A.~Hocquet.
\newblock A semi-discrete scheme for the stochastic {L}andau--{L}ifshitz
  equation.
\newblock {\em Stochastic Partial Differential Equations: Analysis and
  Computations},  {\bf 2} (2014), 281--315.

\bibitem{AloJai06}
F.~Alouges and P.~Jaisson.
\newblock Convergence of a finite element discretization for the
  {L}andau-{L}ifshitz equations in micromagnetism.
\newblock {\em Math. Models Methods Appl. Sci.},  {\bf 16} (2006), 299--316.

\bibitem{Bart05}
S.~Bartels.
\newblock Stability and convergence of finite-element approximation schemes for
  harmonic maps.
\newblock {\em SIAM J. Numer. Anal.},  {\bf 43} (2005), 220--238 (electronic).

\bibitem{BanBrzPro09}
L.~Ba\v{n}as, Z.~Brze{\'z}niak, A.~Prohl, and M.~Neklyudov.
\newblock A convergent finite-element-based discretization of the stochastic
  {L}andau--{L}ifshitz--{G}ilbert equation.
\newblock {\em IMA Journal of Numerical Analysis},  {\bf } (2013).

\bibitem{Brown1979}
W.~Brown.
\newblock Thermal fluctuation of fine ferromagnetic particles.
\newblock {\em IEEE Transactions on Magnetics},  {\bf 15} (1979), 1196--1208.

\bibitem{Brown63}
W.~F. Brown.
\newblock Thermal fluctuations of a single-domain particle.
\newblock {\em Phys. Rev.},  {\bf 130} (1963), 1677--1686.

\bibitem{BrzGolJer12}
Z.~Brze{\'z}niak, B.~Goldys, and T.~Jegaraj.
\newblock Weak solutions of a stochastic {L}andau--{L}ifshitz--{G}ilbert
  equation.
\newblock {\em Applied Mathematics Research eXpress},  {\bf } (2012), 1--33.

\bibitem{Cimrak_survey}
I.~Cimr{\'a}k.
\newblock A survey on the numerics and computations for the {L}andau-{L}ifshitz
  equation of micromagnetism.
\newblock {\em Arch. Comput. Methods Eng.},  {\bf 15} (2008), 277--309.

\bibitem{DaPrato92}
G.~Da~Prato and J.~Zabczyk.
\newblock {\em Stochastic {E}quations in {I}nfinite {D}imensions}.
\newblock Encyclopedia of Mathematics and its Applications. Cambridge
  University Press, Cambridge, 2014.

\bibitem{Doss1977}
H.~Doss.
\newblock Liens entre équations différentielles stochastiques et ordinaires.
\newblock {\em Annales de l'I.H.P. Probabilités et statistiques},  {\bf 13}
  (1977), 99--125.

\bibitem{Gil55}
T.~Gilbert.
\newblock A {L}agrangian formulation of the gyromagnetic equation of the
  magnetic field.
\newblock {\em Phys Rev},  {\bf 100} (1955), 1243--1255.

\bibitem{BNT2016}
B.~Goldys, K.-N. Le, and T.~Tran.
\newblock A finite element approximation for the stochastic
  {L}andau--{L}ifshitz--{G}ilbert equation.
\newblock {\em Journal of Differential Equations},  {\bf 260} (2016), 937 --
  970.

\bibitem{Johnson87}
C.~Johnson.
\newblock {\em Numerical {S}olution of {P}artial {D}ifferential {E}quations by
  the {F}inite {E}lement {M}ethod}.
\newblock Cambridge University Press, Cambridge, 1987.

\bibitem{Chantrelletal2015}
G.~Ju, Y.~Peng, E.~K.~C. Chang, Y.~Ding, A.~Q. Wu, X.~Zhu, Y.~Kubota, T.~J.
  Klemmer, H.~Amini, L.~Gao, Z.~Fan, T.~Rausch, P.~Subedi, M.~Ma,
  S.~Kalarickal, C.~J. Rea, D.~V. Dimitrov, P.~W. Huang, K.~Wang, X.~Chen,
  C.~Peng, W.~Chen, J.~W. Dykes, M.~A. Seigler, E.~C. Gage, R.~Chantrell, and
  J.~U. Thiele.
\newblock High density heat-assisted magnetic recording media and advanced
  characterization --progress and challenges.
\newblock {\em IEEE Transactions on Magnetics},  {\bf 51} (2015), 1--9.

\bibitem{kallenberg}
O.~Kallenberg.
\newblock {\em Foundations of {M}odern {P}robability}.
\newblock Probability and Its Applications. Springer-Verlag New York,
  Cambridge, 2002.

\bibitem{LL35}
L.~Landau and E.~Lifschitz.
\newblock On the theory of the dispersion of magnetic permeability in
  ferromagnetic bodies.
\newblock {\em Phys Z Sowjetunion},  {\bf 8} (1935), 153--168.

\bibitem{sussmann1977}
H.~J. Sussmann.
\newblock An interpretation of stochastic differential equations as ordinary
  differential equations which depend on the sample point.
\newblock {\em Bulletin of the American Mathematical Society},  {\bf 83}
  (1977), 296--298.

\end{thebibliography}
